\documentclass[a4paper,reqno]{amsart}
\usepackage{amsmath,amssymb,amsthm,enumerate}

\usepackage{graphics}
\usepackage{epsfig}
\usepackage{amsfonts}
\usepackage{amssymb}
\usepackage{color}
\usepackage{xcolor}
\usepackage[colorlinks=true]{hyperref}
\hypersetup{linkcolor=red, urlcolor=cyan,citecolor=blue}

\usepackage{tikz}
\usetikzlibrary{arrows, calc}
\usetikzlibrary{calc,quotes,angles}
\usepackage{pgfplots}
\usepackage{mathtools}
\usetikzlibrary{intersections}

\usepackage{booktabs,tabularx}
\usepackage[font=small,labelfont=bf]{caption}

 \usepackage{hyperref}
\usepackage{mathtools}

\usepackage[margin=1in]{geometry}

\def\ifl{\iffalse }

\def\bc{\begin{center}} \def\ec{\end{center}}
\def\ba{\begin{array}} \def\ea{\end{array}}
\def\bea{\begin{eqnarray}} \def\eea{\end{eqnarray}}
\def\beaa{\begin{eqnarray*}} \def\eeaa{\end{eqnarray*}}

\numberwithin{equation}{section}

\theoremstyle{definition}
\newtheorem{thm}{Theorem}[section]

\newtheorem{lem}{Lemma}[section]
\newtheorem{defi}[thm]{Definition}

\theoremstyle{remark}
\newtheorem{rem}{Remark}[section]
\newtheorem*{rem*}{Remark}

\numberwithin{equation}{section}

\newcommand{\R}{\mathbb{R}}

\newcommand{\Z}{\mathbb{Z}}

\renewcommand{\div}{\mathop{\rm div}}

\newcommand{\pa}{\partial}
\newcommand{\na}{\nabla}

\newcommand{\de}{\delta}

\newcommand{\La}{\Lambda}

\newcommand{\Lg}{\langle}
\newcommand{\Rg}{\rangle}
\newcommand{\De}{\Delta}

\newcommand{\ls}{\lesssim}
\newcommand{\T}{\mathbb{T}}
\newcommand{\LP}{\mathbb{P}}

\title[A Besov integration-by-parts method for NS and Euler]{A Besov-based integration-by-parts method for the incompressible Navier-Stokes equations}

\author[X. Cheng]{Xinyu Cheng}
 \address{X. Cheng, Research Institute of Intelligent Complex Systems, Fudan University, Shanghai, P.R. China}
\email{xycheng@fudan.edu.cn}

\author[Z. Luo]{Zhaonan Luo}
\address{Z. Luo, School of Science, Shenzhen Campus of Sun Yat-sen University, Shenzhen, P.R. China}
\email{luozhn7@mail.sysu.edu.cn}

\author[S. Wang]{Sheng Wang}
\address{S. Wang, School of Mathematical Sciences, Peking University, Beijing, P.R. China}
\email{everest\_sheng@pku.edu.cn}

\begin{document}
\maketitle
\begin{abstract}
This note introduces a novel numerical analysis framework for the incompressible Navier-Stokes equations based on Besov spaces. The key contribution of this note is to establish the stability and convergence of a semi-implicit time-stepping scheme by deriving precise error estimates in the $B^0_{\infty,1}$ and $B^0_{\infty,2}$ spaces. Another contribution of our analysis is the detailed treatment of the $B^0_{\infty,2}$ case, where a crucial integration-by-parts technique is employed to adeptly handle the nonlinear advection term. This technique allows for a refined estimate that effectively transfers derivatives onto the test functions, mitigating the inherent analytical challenges posed by the low regularity of these spaces. Our results provide sharper, more localized error bounds than in classical Sobolev spaces, directly linking the scheme's convergence to the critical regularity of the continuous solution. This work underscores the advantage of Besov spaces for the numerical analysis of nonlinear fluid PDEs.
\end{abstract}

\section{Introduction}

The numerical analysis of the incompressible fluid equations remains a topic of profound and enduring interest, standing at the crossroads of theoretical mathematics (partial differential equations and analysis) and computational fluid dynamics (engineering). While the (global) well-posedness theory for these equations is famously challenging, the development of robust and efficient numerical schemes is equally critical for predictive simulations across science and engineering. A vast body of literature is devoted to constructing discretization methods, with a predominant focus on error estimates formulated within the classical Sobolev spaces, such as $L^2(\Omega)$ and $H^s(\Omega)$. These frameworks, though powerful, can sometimes obscure finer-grained information about the solution, particularly concerning its pointwise behavior or spaces closely related to the well-posedness theory of the equations themselves.

In this paper, we introduce a novel numerical analysis framework grounded in the theory of Besov spaces. Specifically speaking, we will perform a rigorous error analysis for a semi-implicit time-stepping scheme applied to the Navier-Stokes equations, deriving error bounds in the $B^0_{\infty,1}$ and $B^0_{\infty,2}$ norms. Here the incompressible Navier-Stokes system (NS) we consider are described as below:
\begin{equation}\label{NSE}
\begin{cases}
&\pa_t u +u \cdot \na u+\na p =\nu \De u,\qquad \na \cdot u=0;\\
&u|_{t=0}=u_0.
\end{cases}
\end{equation}
The solution $u:\R^+\times \Omega\to \R^d$ is the velocity field and $ p :\R^+\times \Omega \to \R$ is the pressure. Moreover, the divergence-free condition $\na \cdot u=0$ is to guarantee the incompressiblity. Since we are interested in the analysis framework, for the sake of implicit we shall only consider the periodic boundary condition: $\Omega=\T^2$. And the semi-implicit scheme we consider is the following:
\begin{equation}\label{1.3}
\begin{cases}
&\frac{u^{n+1}-u^n}{\tau}+(u^{n}\cdot \na) u^{n+1}+\na p^n= \nu \Delta u^{n+1},\\
&\div u^{n+1}=0,
\end{cases}
\end{equation}
where $\tau$ is the time step. Moreover in \eqref{1.3} the viscosity term $\nu \Delta u^{n+1}$ is unknown and the nonlinear term $u^{n}\cdot \na u^{n+1}$ is indeed linear with respect to $u^{n+1}$ at $n+1$-th time step. In addition we can further derive that $\div u^n=0$ for all $n$ by induction.

This choice of functional setting is motivated by several compelling factors. Firstly, Besov spaces like $B^0_{\infty,1}$ are known to be critical for the existence and uniqueness theory of Navier-Stokes and Euler equations; indeed the global well-posedness of the Cauchy problem of the Euler equations in the borderline Besov space $B^{\frac{3}{p}+1}_{p,1}(\R^3)$ for $p\in[1,\infty]$ is known to be open, cf. Bourgain-Li \cite{BL15,BL15b}. These spaces provide a natural and less restrictive environment for analyzing solutions compared to Sobolev spaces, particularly for problems involving rough initial data or intricate vortex structures. Secondly, an error estimate in $B^0_{\infty,\cdot}$ offers a stronger, more localized description over the pointwise error of the numerical solution than an $L^2$-estimate. It is worth mentioning here in a recent work \cite{CLW25}, a numerical method of low regularity assumption was developed to study the surface quasi-geostrophic equations in a $B^0_{\infty,1}$ set-up, which motivates us to develop a more general $B^0_{\infty,1}$ framework for impressible fluid equations such as Navier-Stokes and Euler equations in this presenting note. Moreover, in another recent work \cite{CLW24} we observe that a simple integration-by-parts can help to lower the regularity assumption, which motivates us to study the $B^0_{\infty,2}$ numerical space.

Our work demonstrates that the proposed semi-implicit scheme is stable and convergent within this refined functional framework. The analysis reveals how the regularity of the exact solution, measured in these Besov norms, directly dictates the convergence rates of the numerical method. By bridging the gap between the analytical functional setting of the continuous equations and the discrete numerical analysis, this approach not only provides stronger theoretical guarantees but also offers deeper insight into the structure of the error. We anticipate that this Besov-based framework will open new avenues for the analysis and development of numerical methods for a broader class of nonlinear fluid PDEs where such spaces are intrinsic to the well-posedness theory.

We now give a brief review on the literature of interest. First of all, we refer the well-posedness theory of the Euler and Navier-Stokes equations to \cite{BCD11,Y63,BKM84,L34} and the references therein. We also refer the readers to \cite{CLYY24,LL11} for more discusssion in dfferent models and function spaces. Since we are also interested in developing the vanishing-viscosity-limit structure-preserving schemes, we list here several related results for different fluid models in different function spaces: \cite{CLW25,GLY19,HK08,M07}. The ill-posednes of the different fluid systems have been investigated, cf. \cite{BL15,BL15b,BP08}. The non-uniqueness of different fluid equations have been obtained in \cite{BSV19,CKL21,DLS13} and etc. On the other hand, there are many numerical methods to study the incompressible fluid models including Euler and Navier-Stokes equations. For example, the backward Euler differentiation method cf.\cite{He08,Wang12}, finite element methods cf. \cite{HR82,HR90,HS00, HZ18,V84}, finite difference methods cf. \cite{Chorin68,EL95,EL02,Wetton06}, spectral methods cf. \cite{GZ03,CLW24}, the Lagrange–Galerkin method cf. \cite{He13,S88}, the projection method cf. \cite{Chorin69,HW93,EL95,EL02}, exponential low regularity methods cf. \cite{LMS22,RS21,BLW22,WZ22,CLW25} and neural network based methods cf. \cite{HJE18, JCLK21,KKLPWY21}. We also refer the readers to \cite{MT98,Tem66,Tem24} and the references therein for a more detailed discussion.

Now in pratice to solve $u^{n+1}$ from \eqref{1.3}, we recall the Leray projection $\LP$, namely the $L^2$-orthogonal projection onto the divergence-free subspace: for any $v\in L^2$ we have $\LP v=v-\na q$, where $q\in H^1$ solves the following Poisson equation under periodic boundary conditions:
\begin{equation*}
        \De q=\na \cdot v.
\end{equation*}
Then we apply the Leray projection $\LP$ to derive that 
{\begin{equation}\label{1.5}
\frac{u^{n+1}-u^n}{\tau} +\LP(u^n \cdot \na u^{n+1})=\nu\Delta u^{n+1}.
\end{equation}}
In fact we will solve the NS (Euler) equations from the first order semi-implicit scheme \eqref{1.5}. We shall show our scheme \eqref{1.5} preserves regularity stability; moreover the error will be shown with numerical evidence and rigorous analysis. Our main results state below. To start with, we first present the stability result of the semi-implicit scheme \eqref{1.3} or \eqref{1.5}.

\begin{thm}[``Unconditional stability'']\label{Thm1.1}
Consider the NS equations \eqref{NSE} in $\T^2$. Assume that the initial data is $u_0$ then there exists $T>0$ such that the following statements hold for any $M\in \mathbb{Z}^+$, $n\le M$ and viscosity $\nu\ge 0$. We only require the time step $\tau=\frac{T}{M}\le T$ which is essentially unconditional.
\begin{description}
    \item[(i)] Assume $u_0\in B^{s}_{\infty,1}$ with $s\ge1$, then the $B^{s}_{\infty,1}$-norm is uniformly bounded:
    \begin{equation}\label{1.6}
       \sup_n\|u^n\|_{B^{s}_{\infty,1}}\le 8\|u_0\|_{B^{s}_{\infty,1}}.
    \end{equation}
\item[(ii)]Assume $u_0\in B^{s}_{\infty,2}$ with $s>1$, then the $B^{s}_{\infty,2}$-norm is uniformly bounded:
\begin{equation}\label{1.7}
   \sup_n\|u^n\|_{B^{s}_{\infty,2}}\le 8\|u_0\|_{B^{s}_{\infty,2}}.
\end{equation}
\end{description}

\end{thm}

\begin{rem}
  Since the main purpose of the presenting paper is to provide a Besov framework for solving the NS equations numerically, we will consider peridoic domain for simplicity. In fact our method can be applied to either Dirichlet or Neumann boundary conditions. Our framework can be applied to the no-slip boundary condition as well with careful modifications. In such case, more terms arises due to the commutator of Leray projection and Laplacian (and others).
\end{rem}

The error can be analyzed by the following theorems.

\begin{thm}[$B^s_{\infty,1}$-Error estimates]\label{Thm1.2}
Assume that $u(t,x) $ is the exact solution to \eqref{NSE} with the same initial condition $u_0\in B^s_{\infty,1}(\T^2)$ for $s\ge4$. Then the following error estimates hold.
\begin{description}
    \item[(i)] The $B^0_{\infty,1}$-error estimate holds:
 \begin{equation}\label{1.8a}
   \sup_{n} \|u^{n+1} -u \left( t_{n+1}\right)\|_{B^0_{\infty,1}}\leq C_1\tau ;
 \end{equation}
 
\item[(ii)]  Additionally if $u_0\in B^{s+4}_{\infty,1}(\T^2)$, then the $B^s_{\infty,1}$-error estimate holds:
\begin{equation}\label{1.8b}
   \sup_n\|u^{n+1} - u\left( t_{n+1}\right)\|_{B^s_{\infty,1}}\le  C_2\tau.
\end{equation}
\end{description}
Here the constants $C_1,C_2>0$ above only depend on the initial condition $u_0$. Both constants are independent of the viscosity $\nu$.
\end{thm}

\begin{thm}[$B^s_{\infty,2}$-Error estimates]\label{Thm1.3}
 Assume that $u(t,x) $ is the exact solution to \eqref{NSE} with the same initial condition $u_0\in B^s_{\infty,2}(\T^2)$ for $s\ge3$. Then the following error estimates hold.
\begin{description}
    \item[(i)] The $B^0_{\infty,2}$-error estimate holds:
 \begin{equation}\label{1.9a}
   \sup_{n} \|u^{n+1} -u \left( t_{n+1}\right)\|_{B^0_{\infty,2}}\leq C_1\tau ;
 \end{equation}
 
\item[(ii)]  Additionally if $u_0\in B^{s+3}_{\infty,2}(\T^2)$, then the $B^s_{\infty,1}$-error estimate holds:
\begin{equation}\label{1.9b}
   \sup_n\|u^{n+1} - u\left( t_{n+1}\right)\|_{B^s_{\infty,2}}\le  C_2\tau.
\end{equation}
\end{description}
Here the constants $C_1,C_2>0$ above only depend on the initial condition $u_0$. Both constants are independent of the viscosity $\nu$.
\end{thm}

\begin{rem}
    The high regularity assumption $u_0\in B^4_{\infty,1}$ or  $B^3_{\infty,2}$ is needed only to deal with the viscosity $\nu\De u^{n+1}$ term in the scheme \eqref{1.3}. Indeed one only requires $u_0\in B^2_{\infty,1}$ or  $B^2_{\infty,2}$ if the incompressible Euler equations are to solved; we refer the readers to the details later in Section \ref{sec:Besov1} and \ref{sec:Besov2}.
\end{rem}

Throughout this paper, for any two (non-negative in particular) quantities $X$ and $Y$, we denote $X \lesssim Y$ if
$X \le C Y$ for some constant $C>0$. Similarly $X \gtrsim Y$ if $X
\ge CY$ for some $C>0$. We denote $X \sim Y$ if $X\lesssim Y$ and $Y
\lesssim X$. The dependence of the constant $C$ on
other parameters or constants are usually clear from the context and
we will often suppress  this dependence. We shall denote
$X \lesssim_{Z_1, Z_2,\cdots,Z_k} Y$
if $X \le CY$ and the constant $C$ depends on the quantities $Z_1,\cdots, Z_k$. We define $[A,B]$ to be $AB-BA$, namely the usual commutator.


The presenting paper is organized as follows. In Section~\ref{sec2} we introduce the Little-wood Paley theory and key preliminaries including several useful lemmas. In Section~\ref{sec:stab} we prove the stability Theorem~\ref{Thm1.1}; the $B^0_{\infty,1}$ and $B^0_{\infty,2}$ errors are estimated in Section~\ref{sec:Besov1} and Section~\ref{sec:Besov2}. We give more details about the full discretization of the scheme by discussing more on the finite element methods and Fourier spectral methods in Section~\ref{sec:full}. We will provide numerical evidence in Section~\ref{sec7}.


\section{Littlewood-Paley theory and Besov spaces}\label{sec2}
In this section we give a biref introduction to the Littlewood-Paley theory and Besov spaces that will be often times used. We first define the Littlewood-Paley dyadic decomposition on $\mathbb{T}^{d}$. Assume $\chi: [0,+\infty)\to [0,1]$ is smooth, monotone, and satisfying
 \begin{equation*}
 \chi(x)=
 \begin{cases} 
 1,\quad 0\leq x \leq\frac{1}{2};\\
 \frac{1}{2},\quad x=\frac{3}{4};\\
 0,\quad x\geq 1.
 \end{cases}
 \end{equation*}
 We then define $\varphi(x)=\chi(x)-\chi(2x)$, then
 \begin{equation*}
 \chi(|x|)+\sum_{j\geq 1}\varphi\left(\frac{|x|}{2^{j}}\right)\equiv 1,\quad \forall x\in\mathbb{R}^{d}.
 \end{equation*}
 \begin{defi}
      Assume $ f $ is a distribution on $\mathbb{T}^{d}$ with
 \begin{equation*}
 f(x)= \frac{1}{(2\pi)^d}\sum_{k\in\mathbb{Z}^{d}}\hat{f}(k)e^{2\pi ik\cdot x}.
 \end{equation*}
 We define the Littlewood-Paley operators as follows:
 \begin{equation}\label{aeq2}
 \begin{cases}
 \Delta_{0}f=\sum_{k\in\mathbb{Z}^{d}}\hat{f}(k)\chi\left(|k|\right)e^{2\pi ik\cdot x}=\hat{f}(0);\\
 \Delta_{j}f=\sum_{k\in\mathbb{Z}^{d}}\hat{f}(k)\varphi\left(\frac{|k|}{2^{j}}\right)e^{2\pi ik\cdot x},\ j\geq 1;\\
 S_j f =\De_0 f+\sum_{1\le l\le j}\De_l f,\ j\ge 1.
 \end{cases}
 \end{equation}
 Hence the dyadic decomposition of $f$ is given by
 \begin{equation*}
 f=\sum_{j\geq 0}\Delta_{j}f.
 \end{equation*}
In particular, throughout this paper if we further assume $f$ is of zero mean, then the Littlewood-Paley decomposition of $f$ can be decomposed as below:
    \begin{align}\label{defLP2}
        f = \sum_{j\ge 1}\De_j f.
    \end{align}
\end{defi}

 \begin{defi}
Let $s\in\mathbb{R},\ 1\leq p,r\leq\infty.$ The inhomogeneous Besov spaces $B^s_{p,r}$ is defined by
\begin{equation}\label{defBesov}
    B^s_{p,r}=\{u\in S':\|u\|_{B^s_{p,r}}=\Big\|(2^{js}\|\Delta_j u\|_{L^p})_j \Big\|_{l^r(\mathbb{Z})}<\infty\}.
\end{equation}
Moreover, the homogeneous Besov norm and inhomogeneous Besov norm coincide when the function has zero mean.
\end{defi}

The Bernstein’s inequality is given below.
\begin{lem}[Bernstein's inequality]\label{Berns}
\ 
\begin{description}
    \item[(i)] Let $p \in [1,\infty]$ and $s \in\R$. Then for any $j\in 
\Z$, we have
\begin{equation*}
 C\|\La^s \De_j f\|_{L^p}\le 2^{js}\|\De_j f\|_{L^p}\le C'    \|\La^s \De_j f\|_{L^p}
\end{equation*}
with some constants $C$ and $C'$ depending only on $p$ and $s$.
\item[(ii)] Moreover, for $1\le p\le  q \le \infty$, we have
\begin{equation*}
    \| \De_j f\|_{L^q}\ls 2^{jd(\frac{1}{p}-\frac{1}{q})}\|\De_j f\|_{L^p},
\end{equation*}
where the constant depends only on $p$ and $q$.
\end{description}
\end{lem}
Let $u$ and $v$ be tempered distributions in $\mathcal{S}'$, then the non-homogeneous $\rm paraproduct$ of $v$ by $u$ is defined as follows:
\begin{align*}
    T_u v = \sum_{j} S_{j-1}u\Delta_j v,
\end{align*}
and the non-homogeneous $\rm remainder$ of $u$ and $v$ is defined by
\begin{align*}
    R(u,v) = \sum_{|k-j|\leq1} \Delta_{k} u \Delta_{j} v\triangleq\sum_{k\geq0} \Delta_{k}u\tilde{\Delta}_k v.
\end{align*}
At least formally, we obtain the so-called Bony's decomposition:
\begin{align*}
    uv=T_u v + T_v u + R(u,v).
\end{align*}

\begin{lem}\label{T}\cite{BCD11}
    For any $(s,t)\in\mathbb{R}\times(-\infty,0)$ and $(p,p_1,p_2,r,r_1,r_2)\in[1,\infty]^6$, there exists a constant $C$ such that
    \begin{align*}
        \|T_u v\|_{B^s_{p,r}} \leq C^{1+|s|} \|u\|_{L^{p_1}}\|v\|_{B^s_{p_2,r}},
    \end{align*}
    with $(u,v)\in L^{p_1}\times B^s_{p_2,r}$ and $\frac{1}{p}=\frac{1}{p_1}+\frac{1}{p_2}$. Moreover, we have
    \begin{align*}
        \|T_u v\|_{B^{s+t}_{p,r}} \leq \frac{C^{1+|s+t|}}{-t} \|u\|_{B^{t}_{\infty,r_1}}\|v\|_{B^s_{p,r_2}},
    \end{align*}
    with $(u,v) \in B^{t}_{\infty,r_1}\times B^s_{p,r_2}$ and $\frac{1}{r}=\min\{1,\frac{1}{r_1}+\frac{1}{r_2}\}$.
\end{lem}
\begin{lem}\label{R}\cite{BCD11}
    A constant $C$ exists which satisfies the following inequalities. Let $(s_1,s_2)\in\mathbb{R}^2$ and $(p_1,p_2,r_1,r_2) \in [1,\infty]^4$. Assume that $$\frac{1}{p}=\frac{1}{p_1}+\frac{1}{p_2}\leq1~~~~~\text{and}~~~~~\frac{1}{r}=\frac{1}{r_1}+\frac{1}{r_2}\leq1.$$
    If $s_1+s_2>0$, for any $(u,v)\in B^{s_1}_{p_1,r_1}\times B^{s_2}_{p_2,r_2}$, then we have
    \begin{align*}
        \|R(u,v)\|_{B^{s_1+s_2}_{p,r}} \leq \frac{C^{1+|s_1+s_2|}}{s_1+s_2} \|u\|_{B^{s_1}_{p_1,r_1}}\|v\|_{B^{s_2}_{p_2,r_2}}.
    \end{align*}
    If $r=1$ and $s_1+s_2=0$, for any $(u,v)\in B^{s_1}_{p_1,r_1}\times B^{s_2}_{p_2,r_2}$, then we have
    \begin{align*}
        \|R(u,v)\|_{B^{0}_{p,\infty}} \leq C \|u\|_{B^{s_1}_{p_1,r_1}}\|v\|_{B^{s_2}_{p_2,r_2}}.
    \end{align*}
\end{lem}
The following commutator estimates play very important roles in the presenting manuscript.


\begin{lem}[Besov commutator estimate]\label{CI}
Let $1\leq p\leq\infty,\ 1\leq r\leq\infty$ and $\rm div~u=0$. Define $I_j=[u\cdot\nabla, \LP\Delta_j]f$. Then there exists a constant $C$ such that if $\sigma>0$ we have
$$\Big\|(2^{j\sigma}\|I_j\|_{L^p_{x}})_j\Big\|_{l^r(\mathbb{Z})}\leq C(\|\nabla u\|_{L^{\infty}}\|f\|_{B^{\sigma}_{p,r}}+\|\nabla u\|_{B^{\sigma-1}_{p,r}}\|\nabla_{x}f\|_{L^{\infty}_{x}});$$
if $\sigma>1+\frac d p$ or $\sigma=1+\frac d p,~r=1$, we have
$$\Big\|(2^{j\sigma}\|I_j\|_{L^p_{x}})_j\Big\|_{l^r(\mathbb{Z})}\leq C\|\nabla u\|_{B^{\sigma-1}_{p,r}}\|f\|_{B^\sigma_{p,r}};$$
if $\sigma=1+\frac d p,~r>1$, we have
$$\Big\|(2^{j\sigma}\|I_j\|_{L^p_{x}})_j\Big\|_{l^r(\mathbb{Z})}\leq C\|\nabla u\|_{B^{\sigma}_{p,r}}\|f\|_{B^\sigma_{p,r}};$$
if $\sigma<1+\frac d p$, we have
$$\Big\|(2^{j\sigma}\|I_j\|_{L^p_{x}})_j\Big\|_{l^r(\mathbb{Z})}\leq C\|\nabla u\|_{B^{\frac {d} {p}}_{p,r}\cap L^{\infty}}\|f\|_{B^\sigma_{p,r}}.$$
\end{lem}



\begin{proof}[Proof of Lemma~\ref{Berns}-Lemma~\ref{CI}]
    We refer the readers to \cite{BCD11} and \cite{Li19} for the proof. 
\end{proof}
We have the following product rules.
\begin{lem}[Product rules]\label{PL}
For nonlinear terms that satisfy the div-curl structure, i.e. $\div u=0$, the following $B^0_{\infty,1}$ estimate holds 
$$ \|u\cdot \nabla v\|_{B^0_{\infty,1}}\leq C\|u\|_{B^0_{\infty,1}}\|v\|_{B^1_{\infty,1}},$$
for some absolute constant $C>0$.
\end{lem}
\begin{proof}[Proof of Lemma~\ref{PL}]
    Using Bony's decomposition, there holds 
    $$u\cdot \na v=T_u \na v+T_{\na v} u+R(u,\na v).$$
    By Lemma \ref{T}, we deduce that
    \begin{align*}
        \|T_u \na v+T_{\na v} u\|_{B^0_{\infty,1}}\lesssim \|u\|_{B^0_{\infty,1}}\|\na v\|_{B^0_{\infty,1}}.
    \end{align*}
    By virtue of Lemma \ref{R} and the fact that $u\cdot \na v=\mathrm{div\,}(v\otimes u)$, we obtain
    \begin{align*}
        \|R(u,\na v)\|_{B^0_{\infty,1}}\lesssim \|R(u,v)\|_{B^1_{\infty,1}}\lesssim\|u\|_{B^0_{\infty,1}}\|v\|_{B^1_{\infty,1}}.
    \end{align*}
    We thus finish the proof of Lemma \ref{PL}.
\end{proof}

\begin{rem}\label{rem2.1}
    We would like to highlight here that in general the $B^0_{\infty,1}$ space is not a multiplicative algebra, namely, $\|f\cdot g\|_{B^{0}_{\infty,1}}$ cannot be bounded by 
$\|f\|_{B^0_{\infty,1}}\cdot \|g\|_{B^0_{\infty,1}}$. However if they satisfy the div-curl structure as in Lemma~\ref{PL} above, then a similar version holds: $$\|f\cdot \nabla g\|_{B^0_{\infty,1}}\ls\|f\|_{B^0_{\infty,1}}\|g\|_{B^1_{\infty,1}}. $$
Furthermore, if we in addition assume both $f$ and $g$ are of zero mean in $\T^2$, then such subspace is an algebra:
 $$\|f\cdot \nabla g\|_{B^0_{\infty,1}}\ls\|f\|_{B^0_{\infty,1}}\|\na g\|_{B^0_{\infty,1}}. $$
 {One should also note that $B^0_{\infty,2}$ is not a multiplicative algebra either.}
\end{rem}

As a corollary of Lemma~\ref{PL} and Remark~\ref{rem2.1}, the following lemma holds.
\begin{lem}\label{PL1}
For nonlinear terms that satisfy the div-curl structure, i.e. $\div u=0$, the following $B^1_{\infty,1}$ estimate holds for some absolute constant $C>0$
$$ \|u\cdot \nabla v\|_{B^1_{\infty,1}}\leq C\left(\|u\|_{B^0_{\infty,1}}\|v\|_{B^2_{\infty,1}}+\|u\|_{B^1_{\infty,1}}\|v\|_{B^1_{\infty,1}}\right). $$
\end{lem}
\begin{proof}[Proof of Lemma~\ref{PL1}]
   This lemma is a direct corollary of Lemma~\ref{PL} hence we omit the details. 
\end{proof}

Note that the following embedding will be applied oftentimes.

\begin{lem}\label{embed}
    Assume $f$ is smooth in $\T^2$ then it follows for any $s\ge 0$ and $\epsilon>0$ that 
    \begin{equation}
        \|f\|_{B^{s}_{\infty,1}}\le C_\epsilon \|f\|_{B^{s+\epsilon}_{\infty,2}}
    \end{equation}
    for some constant $C_{s,\epsilon}>0$.
\end{lem}
\begin{proof}[Proof of Lemma~\ref{embed}]
We sketch the proof here. 
\begin{align*}
      \|f\|_{B^{s}_{\infty,1}} = \sum_{j\ge -1} 2^{(s+\epsilon)j}\|\De_j f\|_{L^\infty} 2^{-\epsilon j}\lesssim \Big\|(2^{(s+\epsilon)j}\|\Delta_j f\|_{L^\infty})_j \Big\|_{l^2}\cdot \|2^{-\epsilon j}\|_{l^2}\lesssim \|f\|_{B^{s+\epsilon}_{\infty,2}}.
\end{align*}

\end{proof}

\begin{lem}[2nd product rules]\label{PL2}
For nonlinear terms the following $B^0_{\infty,2}$ estimate holds for any $\epsilon>0$:
$$ \|u\cdot \nabla v\|_{B^0_{\infty,2}}\leq C_\epsilon\|u\|_{B^0_{\infty,2}}\|v\|_{B^{1+\epsilon}_{\infty,2}},$$
for some constant $C_\epsilon>0$.
\end{lem}
\begin{proof}[Proof of Lemma~\ref{PL2}]
    Using Bony's decomposition, there holds 
    $$u\cdot \na v=T_u \na v+T_{\na v} u+R(u,\na v).$$
    By Lemmas \ref{T} and \ref{embed}, we deduce that given $\epsilon>0$, it holds that
    \begin{align*}
         \|T_u \na v\|_{B^0_{\infty,2}}\lesssim \|T_u \na v\|_{B^0_{\infty,1}}\lesssim\|u\|_{B^{-\epsilon}_{\infty,2}}\|\na v\|_{B^{\epsilon}_{\infty,2}}\lesssim \|u\|_{B^{0}_{\infty,2}}\|v\|_{B^{1+\epsilon}_{\infty,2}},
    \end{align*}
    and
    \begin{align*}
        \|T_{\na v} u\|_{B^0_{\infty,2}}\lesssim \|\na v\|_{L^\infty}\|u\|_{B^0_{\infty,2}}\lesssim \|u\|_{B^{0}_{\infty,2}}\|v\|_{B^{1}_{\infty,1}}\lesssim\|u\|_{B^{0}_{\infty,2}}\|v\|_{B^{1+\epsilon}_{\infty,2}}.
    \end{align*}
    Similarly, it follows by Lemma \ref{R} that
    \begin{align*}
        \|R(u,\na v)\|_{B^0_{\infty,2}}\lesssim 
        \|R(u,\na v)\|_{B^\epsilon_{\infty,1}}\lesssim
        \|u\|_{B^0_{\infty,2}}\|\na v\|_{B^{\epsilon}_{\infty,2}}\lesssim\|u\|_{B^{0}_{\infty,2}}\|v\|_{B^{1+\epsilon}_{\infty,2}}.
    \end{align*}
    We thus finish the proof of Lemma \ref{PL2}.
\end{proof}

We introduce the following lemma which will be useful in dealing with the pressure term $\nabla p$.
\begin{lem}\label{P}
For any $-1<\sigma\leq 1$, there exists a constant $C$ such that if $\rm div~u=\rm div~v=0$, then
\begin{align*}
\|\nabla(-\Delta)^{-1}\rm div(u\cdot\nabla v)\|_{B^\sigma_{\infty,1}}\leq C\min(\|u\|_{B^\sigma_{\infty,1}}\|v\|_{B^1_{\infty,1}},\|u\|_{B^1_{\infty,1}}\|v\|_{B^\sigma_{\infty,1}}).
\end{align*}
\end{lem}

\begin{proof}[Proof of Lemma~\ref{P}]
    We refer the readers to \cite{CLYY24} and \cite{BCD11} for the details.
\end{proof}

\section{Unconditional stability of the scheme in Besov spaces}\label{sec:stab}
In this section we first show the stability of the semi-implicit scheme as in Theorem~\ref{Thm1.1}. Recall that the semi-implicit scheme
\begin{equation}\label{Semischeme1}
\frac{u^{n+1}-u^n}{\tau}+\left(u^{n}\cdot \na u^{n+1}\right)+\na p^n= \nu \Delta u^{n+1}, \quad \nabla\cdot u^{n+1}=0
\end{equation}
where $\tau$ is the size of the time step and we define $M= \frac{T}{\tau}$, the total number of steps. To solve $u^{n+1}$, we apply the Leray projection $\LP$  to  the equation \eqref{Semischeme1} to obtain that 
\begin{equation}\label{Semischeme1b}
\frac{u^{n+1}-u^n}{\tau} +\LP(u^n \cdot \na u^{n+1})=\nu \Delta u^{n+1}.
\end{equation}
Then we apply the Littlewood-Paley projection $\De_j$ to \eqref{Semischeme1b} (note that $\De_j$ commutes with both $\LP$ and $\De$) and thus obtain that
\begin{equation}\label{Semischeme2a}
\frac{\De_j u^{n+1}-\De_j u^n}{\tau} +\LP \De_j (u^n \cdot \na u^{n+1})=\nu \Delta \De_j u^{n+1},
\end{equation}
or equivalently
\begin{equation}\label{Semischeme2}
\frac{\De_j u^{n+1}-\De_j u^n}{\tau} + (u^n \cdot \na \LP\De_j u^{n+1})-\left([u^n\cdot \na, \LP\De_j]u^{n+1}\right)=\nu \Delta \De_j u^{n+1}.
\end{equation}
Then we multiply \eqref{Semischeme2} by $|\De_j u^{n+1}|^{p-2}\De_j u^{n+1} $ and integrate. Note that by the divergence free condition $\LP \De_j u^{n+1} = \De_j u^{n+1}$, we thus have
\begin{align*}
    \Lg u^n\cdot \na \LP \De_j u^{n+1}, |\De_j u^{n+1}|^{p-2}\De_j u^{n+1}  \Rg = 0.
\end{align*}
Using Plancherel formula and noting that the frequency of $\Delta_j$ is supported in a circular annulus with radius approximately $2^j$, we infer that 
\begin{equation}\label{Planch}
\begin{aligned}
    \Lg \De\De_j e^{n+1} , |\De_j e^{n+1}|^{p-2} \De_j e^{n+1} \Rg &=\frac{1}{(2\pi)^2} \Lg \mathcal{F} \left(\De\De_j e^{n+1} \right), \mathcal{F} \left( |\De_j e^{n+1}|^{p-2} \De_j e^{n+1} \right ) \Rg \\
    &=\frac{1}{(2\pi)^2} \Lg  -|\xi|^2  \mathcal{F} \left( \De_je^{n+1} \right), \mathcal{F} \left( |\De_j e^{n+1}|^{p-2} \De_j e^{n+1} \right ) \Rg\\
    &\le -\frac{2^{2j}}{(2\pi)^2} \Lg  \mathcal{F} \left( \De_je^{n+1} \right), \mathcal{F} \left( |\De_j e^{n+1}|^{p-2} \De_j e^{n+1} \right ) \Rg\\
    & = -2^{2j} \Lg \De_j e^{n+1} , |\De_j e^{n+1}|^{p-2} \De_j e^{n+1} \Rg\\
    & = -2^{2j}\|\De_j e^{n+1}\|^p_{L^p}.
\end{aligned}    
\end{equation}
Then it follows by H\"older's inequality that
\begin{align*}
    \|\De_j u^{n+1}\|_{L^p}^p+2^{2j}\nu\tau\|\De_j u^{n+1}\|_{L^p}^p\le  \|\De_j u^{n}\|_{L^p}\|\De_j u^{n+1}\|_{L^p}^{p-1}+\widetilde{C}\tau \|[u^n\cdot \na, \LP\De_j]u^{n+1}\|_{L^p}\|\De_j u^{n+1}\|_{L^p}^{p-1},
\end{align*}
which implies
\begin{align*}
 \|\De_j u^{n+1}\|_{L^p}+2^{2j}\nu\tau\|\De_j u^{n+1}\|_{L^p}
   \le \|\De_j u^{n}\|_{L^p}+\widetilde{C}\tau \|[u^n\cdot \na, \LP\De_j]u^{n+1}\|_{L^p}.
\end{align*}

\texttt{Part (i).} We now show the $B^s_{\infty,1}$ stability for $s\ge 1$. Inductively, we assume $\|u^n\|_{B^{s}_{\infty,1}}\le 8\|u_0\|_{B^{s}_{\infty,1}}$. Then we obtain 
\begin{align*}
    2^{sj} \|\De_j u^{n+1}\|_{L^p}+C_s2^{(2+s)j}\nu\tau\|\De_j u^{n+1}\|_{L^p}
   \le 2^{sj} \|\De_j u^{n}\|_{L^p}+\widetilde{C_s}2^{sj}\tau \|[u^n\cdot \na, \LP\De_j]u^{n+1}\|_{L^p}.
\end{align*}
Then by letting $p\to\infty$ and summing in $j$ we obtain that 
\begin{equation}
\| u^{n+1}\|_{B^s_{\infty,1}}\le \| u^{n}\|_{B^s_{\infty,1}}+\left\| 2^{sj}\left(\|[u^n\cdot \na, \LP\De_j]u^{n+1}\|_{L^\infty}\right)_j\right\|_{l^1}.
\end{equation}
By Lemma \ref{CI} with $s\ge 1$, we know that
\begin{equation}
\| u^{n+1}\|_{B^s_{\infty,1}}\le \| u^{n}\|_{B^s_{\infty,1}}+C_s\tau \|u^n\|_{B^{s}_{\infty,1}}\|u^{n+1}\|_{B^{s}_{\infty,1}}.
\end{equation}
Here $C_s>0$ only depends on $s\ge 1$.

\begin{rem}\label{rem3.1}
   The local existence theory of nonlinear PDEs  usually assumes sufficiently small time interval $T$. Similar to the discussion in \cite{CLW24}, without loss of generality we can take the time interval $T= \frac{1}{8 C_s\|u_0\|_{B^{s}_{\infty,1}}}$. By copying the time interval repeatedly, we can extend the time interval to the whole life-span $T_{\text{max}}$. Whether $T_{\text{max}}=\infty$ is open for the three dimensional incompressible NS equations.
\end{rem}

As discussed above in Remark~\ref{rem3.1}, there is no loss to assume $T= \frac{1}{8 C_s\|u_0\|_{B^{s}_{\infty,1}}}$ and $\|u^n\|_{B^{s}_{\infty,1}} \leq 8\|u_0\|_{B^{s}_{\infty,1}}$. We thus get that
\begin{equation*}
\begin{aligned}
   \|u^{n+1}\|_{B^{s}_{\infty,1}} &\leq \|u^n\|_{B^{s}_{\infty,1}} + C_s\tau \|u^n\|_{B^{s}_{\infty,1}}\|u^{n+1}\|_{B^{s}_{\infty,1}}\\
   &\leq \|u^n\|_{B^{s}_{\infty,1}} + 8C_s\frac{T}{M} \|u_0\|_{B^{s}_{\infty,1}}\|u^{n+1}\|_{B^{s}_{\infty,1}}\\
   &\leq \|u^n\|_{B^{s}_{\infty,1}} + \frac{1}{M} \|u^{n+1}\|_{B^{s}_{\infty,1}},
\end{aligned}
\end{equation*}
which implies that
\begin{align*}
\|u^{n+1}\|_{B^{s}_{\infty,1}} \leq (1+h)\|u^n\|_{B^{s}_{\infty,1}},
\end{align*}
where $h:= \frac{1}{M-1}\le 1$. Iterating the inequality above we can complete the proof: 
\begin{align*}
    \|u^{n+1}\|_{B^{s}_{\infty,1}} &\leq (1+h)^{n+1}\|u_0\|_{B^{1}_{\infty,1}}\nonumber\\
    &\leq (1+h)^{M+1}\|u_0\|_{B^{s}_{\infty,1}}\nonumber\\
    &\leq (1+h)^{\frac{1}{h}+2}\|u_0\|_{B^{s}_{\infty,1}}\nonumber\\
    &\leq 8\|u_0\|_{B^{s}_{\infty,1}}.
\end{align*}
Here the last inequality follows from the fact that $(1+h)^{2+\frac{1}{h}}\le 8$ for $h<1$. We then conclude \eqref{1.6}.

\texttt{Part (ii).} We now show thte $B^s_{\infty,2}$ stability $(s> 1)$.

Inductively, we assume $\|u^n\|_{B^{s}_{\infty,2}}\le 8\|u_0\|_{B^{s}_{\infty,2}}$. We have
\begin{align*}
    2^{sj} \|\De_j u^{n+1}\|_{L^p}+C_s2^{(2+s)j}\nu\tau\|\De_j u^{n+1}\|_{L^p}
   \le 2^{sj} \|\De_j u^{n}\|_{L^p}+\widetilde{C_s}2^{sj}\tau \|[u^n\cdot \na, \LP\De_j]u^{n+1}\|_{L^p}.
\end{align*}
Then by letting $p\to\infty$ and taking $l^2$-norm for  $j$ we obtain that 
\begin{equation}
\begin{aligned}
\| u^{n+1}\|_{B^s_{\infty,2}}\le \| u^{n}\|_{B^s_{\infty,2}}+C_s\tau \left\| 2^{sj}\left(\|[u^n\cdot \na, \LP\De_j]u^{n+1}\|_{L^\infty}\right)_j\right\|_{l^2}.
\end{aligned}
\end{equation}
By Lemma \ref{CI}, we have
\begin{equation}
\begin{aligned}
\| u^{n+1}\|_{B^s_{\infty,2}}\le& \| u^{n}\|_{B^s_{\infty,2}}+C_s\tau(\|\na u^n\|_{L^\infty} \|u^{n+1}\|_{B^s_{\infty,2}}+\|\na u^n\|_{B^{s-1}_{\infty,2}} \|\na u^{n+1}\|_{L^\infty})\\
\le &\| u^{n}\|_{B^s_{\infty,2}} +C_s \tau (\|u^n\|_{B^1_{\infty,1}} \|u^{n+1}\|_{B^s_{\infty,2}}+\| u^n\|_{B^{s}_{\infty,2}} \| u^{n+1}\|_{B^1_{\infty,1}})\\
\le& \| u^{n}\|_{B^s_{\infty,2}} +C_s \tau \|u^n\|_{B^s_{\infty,2}} \|u^{n+1}\|_{B^s_{\infty,2}},
\end{aligned}
\end{equation}
where the last inequality follows from Lemma~\ref{embed} noting $s>1$. Here $C_s>0$ only depends on $s>1$. By Remark~\ref{rem3.1}, we can assume $T= \frac{1}{8 C_s\|u_0\|_{B_{\infty, 2}^s}}$ from the local theory and $\|u^n\|_{B_{\infty, 2}^s} \leq 8 \|u_0\|_{B_{\infty, 2}^s}$. Similarly we get that
\begin{equation*}
\begin{aligned}
   \|u^{n+1}\|_{B_{\infty, 2}^s} &\leq \|u^n\|_{B_{\infty, 2}^s} + C_s\tau \|u^n\|_{B_{\infty, 2}^s}\|u^{n+1}\|_{B_{\infty, 2}^s}\\
   &\leq \|u^n\|_{B_{\infty, 2}^s} + 8C_s\frac{T}{M} \|u_0\|_{B_{\infty, 2}^s}\|u^{n+1}\|_{B_{\infty, 2}^s}\\
   &\leq \|u^n\|_{B_{\infty, 2}^s} + \frac{1}{M} \|u^{n+1}\|_{B_{\infty, 2}^s},
\end{aligned}
\end{equation*}
which implies that
\begin{align*}
\|u^{n+1}\|_{B_{\infty, 2}^s} \leq (1+h)\|u^n\|_{B_{\infty, 2}^s},
\end{align*}
where $h:= \frac{1}{M-1}$. This leads to
\begin{align*}
    \|u^{n+1}\|_{B_{\infty, 2}^s} \leq 8\|u_0\|_{B_{\infty, 2}^s}.
\end{align*}
This then proves \eqref{1.7}. Note that here we only requires $h=\frac{1}{M-1}\le 1$, or equivalently $\tau\le T$, a very mild restriction.

\begin{rem}
    Note that following embeddings hold: $B^{0}_{\infty,1}\hookrightarrow L^\infty$ and $B^{0}_{\infty,1}\hookrightarrow B^{0}_{\infty,2}$; however no inclusion relation holds between $L^\infty$ and $B^{0}_{\infty,2}$.
\end{rem}

\section{$B^0_{\infty,1 }$ error estimates}\label{sec:Besov1}
In this section we shall prove the $B^0_{\infty,1 }$ error estimates by proving Theorem~\ref{Thm1.2}. Assume $u(t)$ is the exact solution to the NS equations \eqref{NSE}. Then by the fundamental theorem of calculus, we have
\begin{align*}
    u\left(t_{n+1}\right)&=u\left(t_{n}\right)+\int^{t_{n+1}}_{t_n} \partial_{t‘}u dt'\nonumber\\
    &= u\left(t_{n}\right)+\int^{t_{n+1}}_{t_n}  - \LP(u \cdot \na u) + \nu \De u\, dt'.
\end{align*}
This together with \eqref{Semischeme2} imply that 
\begin{equation}\label{4.1}
\begin{aligned}
  &\De_j  u^{n+1}-  \De_j  u\left(t_{n+1}\right) \\
  =&  \Delta_j  u^n - \tau   \LP \De_j (u^{n} \cdot \na u^{n+1}) + \tau \nu \De  \De_j  u^{n+1} -  \De_j u\left(t_n\right) +\int^{t_{n+1}}_{t_n}   \LP\De_j(u \cdot \na u) - \nu \De \De_j u \, dt'\\
    =&\De_ju^n-\De_ju\left(t_n\right)+ \nu \int^{t_{n+1}}_{t_n} \De \De_ju^{n+1}- \De \De_j u\, dt' +  \int^{t_{n+1}}_{t_n} \LP \De_j(u \cdot \na u - u^{n} \cdot \na u^{n+1})\, dt'.
\end{aligned}
\end{equation}
Then, we consider the $L^p$-estimates for $\De_j e^{n+1}:=\De_j u^{n+1}-\De_j u\left(t_{n+1}\right)$. We multiply \eqref{4.1} by $|\De_j e^{n+1}\|^{p-2}\De_j e^{n+1}$ and integrate: 
\begin{equation}\label{4.2}
\begin{aligned}
    &\Lg \De_j e^{n+1}, |\De_j e^{n+1}|^{p-2} \De_j e^{n+1}  \Rg \\&= \Lg \De_j e^{n},|\De_j e^{n+1}|^{p-2} \De_j e^{n+1}\Rg\\
    &\quad + \nu \int^{t_{n+1}}_{t_n} \Lg \De \De_j u^{n+1}- \De \De_j u , |\De_j e^{n+1}|^{p-2} \De_j e^{n+1} \Rg dt' \\
    &\quad+  \int^{t_{n+1}}_{t_n} \Lg \LP \De_j (u \cdot \na u - u^{n} \cdot \na u^{n+1}), |\De_j e^{n+1}|^{p-2} \De_j e^{n+1} \Rg dt'\\
    &:= I_1 + I_2 + I_3.
    \end{aligned}
    \end{equation}
It follows by H\"older's inequality that
\begin{align}\label{I1}
    I_1 \leq \|\De_j e^{n+1}\|^{p-1}_{L^p}\|\De_j e^{n}\|_{L^p}.
\end{align}
For $I_2$, we have
\begin{align*}
    I_2 &=\nu \int^{t_{n+1}}_{t_n} \Lg \De\De_j u^{n+1}- \De\De_j u\left(t_{n+1}\right),  |\De_j e^{n+1}|^{p-2} \De_j e^{n+1} \Rg\, dt'\nonumber\\
    &\quad+ \nu \int^{t_{n+1}}_{t_n} \Lg \De\De_j u\left(t_{n+1}\right)- \De\De_j u, |\De_j e^{n+1}|^{p-2} \De_j e^{n+1} \Rg\, dt' \nonumber\\
    &\coloneqq I_{2,1}+I_{2,2}.
\end{align*}
Note that

\begin{align*}
    I_{2,1}\coloneqq & \nu \int^{t_{n+1}}_{t_n} \Lg \De\De_j u^{n+1}- \De\De_j u\left(t_{n+1}\right), |\De_j e^{n+1}|^{p-2} \De_j e^{n+1} \Rg \, dt'\\
    =& \nu\tau \Lg \De\De_j e^{n+1} , |\De_j e^{n+1}|^{p-2} \De_j e^{n+1} \Rg\\.
\end{align*}
It follows by \eqref{Planch} that
\begin{equation}
     I_{2,1}\le-\nu\tau2^{2j}\|\De_j e^{n+1}\|^p_{L^p}\le 0.
\end{equation}
Similarly we can estimate $I_{2,2}$.
\begin{align*}
     I_{2,2} &= -\nu 2^{2j} \int^{t_{n+1}}_{t_n} \Lg \De_j u\left(t_{n+1}\right)- \De_j u, |\De_j e^{n+1}|^{p-2} \De_j e^{n+1} \Rg\, dt'\nonumber\\
     &\leq \nu 2^{2j} \|\De_j e^{n+1}\|^{p-1}_{L^p}  \int^{t_{n+1}}_{t_n} \|\De_j u\left(t_{n+1}\right)- \De_j u\|_{L^p} \, dt'\nonumber\\
     &\leq \nu 2^{2j} \|\De_j e^{n+1}\|^{p-1}_{L^p}  \int^{t_{n+1}}_{t_n} \|\int^{t_{n+1}}_{t'} \partial_t \De_j u\, ds \|_{L^p}\, dt'\nonumber\\
     & \leq \nu 2^{2j} \|\De_j e^{n+1}\|^{p-1}_{L^p}   \int^{t_{n+1}}_{t_n} \int^{t_{n+1}}_{t'} \| \partial_t \De_j u  \|_{L^p} \, ds dt'. \nonumber
\end{align*}
Noting that by the Leray projected PDE \eqref{NSE} we have 
\begin{equation*}
    \pa_t \De_j u =-\LP\Delta_j (u\cdot \na u)+\nu\Delta\Delta_j u.
\end{equation*}
Therefore we can derive that
\begin{equation}
         I_{2,2}\leq  \nu \tau^2 2^{2j}  \left( \|\De_j \left( u\cdot \nabla u\right)\|_{L^\infty_tL^p} + \nu \|\De_j \De u\|_{L^\infty_tL^p}\right) \|\De_j e^{n+1}\|^{p-1}_{L^p} . \nonumber
\end{equation}
It then follows that
\begin{align}\label{I2}
     I_{2} \leq \nu \tau^2  2^{2j}  \left( \|\De_j \left( u\cdot \nabla u\right)\|_{L^\infty_tL^p} + \nu \|\De_j \De u\|_{L^\infty_tL^p}\right) \|\De_j e^{n+1}\|^{p-1}_{L^p} . 
\end{align}
Notice that we have
\begin{equation}
\begin{aligned}\label{chaxiang}
    u\cdot\nabla u - u^n\cdot\nabla u^{n+1}&= u\cdot \nabla u - u\left(t_n\right)\cdot\nabla u\\
    &\quad+ u\left(t_n\right)\cdot\nabla u - u\left(t_n\right)\cdot\nabla u\left(t_{n+1}\right)\\
    &\quad+u\left(t_n\right)\cdot\nabla u\left(t_{n+1}\right)-u^n\cdot\nabla u\left(t_{n+1}\right)\\
    &\quad+ u^n\cdot\nabla u\left(t_{n+1}\right) - u^n\cdot\nabla u^{n+1}.
\end{aligned}
\end{equation}
Then we can rewrite $I_3$ as follows
\begin{align*}
    I_3&= \int^{t_{n+1}}_{t_n} \Lg  \LP \De_j \left ( \left( u-u\left(t_n\right)\right)\cdot \nabla u\right),  |\De_j e^{n+1}|^{p-2} \De_j e^{n+1}\Rg \,dt'\\
    &\quad+\int^{t_{n+1}}_{t_n} \Lg  \LP \De_j \left ( u\left(t_n\right)\cdot \nabla \left(u-u\left(t_{n+1}\right)\right)\right), |\De_j e^{n+1}|^{p-2} \De_j e^{n+1} \Rg \,dt' \\
    &\quad+ \int^{t_{n+1}}_{t_n} \Lg  \LP \De_j  \left ( \left( u\left(t_n\right)-u^n\right)\cdot \nabla u\left(t_{n
    +1}\right)\right), |\De_j e^{n+1}|^{p-2} \De_j e^{n+1} \Rg \,dt'\\
    &\quad+  \int^{t_{n+1}}_{t_n}   \Lg  \LP \De_j  \left (  u^n \cdot \nabla \left (u\left(t_{n
    +1}\right)-u^{n+1}\right)\right),|\De_j e^{n+1}|^{p-2} \De_j e^{n+1}\Rg \,dt'\\
    &:= I_{3,1} + I_{3,2} + I_{3,3} + I_{3,4}
\end{align*} 
Using \eqref{NSE} and the uniform estimates \eqref{1.6}, we have
\begin{equation}\label{I31}
\begin{aligned}
    I_{3,1}  =& \int^{t_{n+1}}_{t_n} \Lg   \left ( \left( u-u\left(t_n\right)\right)\cdot \nabla \LP \De_j u\right),  |\De_j e^{n+1}|^{p-2} \De_j e^{n+1}\Rg \,dt'\\
    &- \int^{t_{n+1}}_{t_n}  \Lg   \left[ \left( u-u\left(t_n\right)\right)\cdot \nabla, \LP \De_j \right]u,  |\De_j e^{n+1}|^{p-2} \De_j e^{n+1}\Rg \, dt'\\
    \leq &\tau \left(\|u-u\left(t_n\right)\|_{L^\infty_tL^\infty}\| \nabla\De_j u\|_{L^\infty_tL^p} + \|\left[ \left( u-u\left(t_n\right)\right)\cdot \nabla, \LP \De_j \right]u\|_{L^\infty_tL^p}\right)\|\De_j e^{n+1}\|^{p-1}_{L^p}\\
    \leq &C \tau^2  2^j \left(\nu \|u_0\|_{B^2_{\infty,1}} +\|u_0\|^2_{B^2_{\infty,1}} \right)\|\De_j u\|_{L^\infty_tL^p} \|\De_j e^{n+1}\|^{p-1}_{L^p}\\
    &+ C\tau \|\left[ \left( u-u\left(t_n\right)\right)\cdot \nabla, \LP \De_j \right]u\|_{L^\infty_tL^p}\|\De_j e^{n+1}\|^{p-1}_{L^p},
 \end{aligned}    
\end{equation}
and
\begin{equation}\label{I32}
    \begin{aligned}
I_{3,2} =& \int^{t_{n+1}}_{t_n} \Lg   \left ( u\left(t_n\right)\cdot \nabla \LP \De_j \left(u-u\left(t_{n+1}\right)\right)\right), |\De_j e^{n+1}|^{p-2} \De_j e^{n+1} \Rg\, dt' \\
&- \int^{t_{n+1}}_{t_n} \Lg  \left[ u\left(t_n\right)\cdot \nabla, \LP \De_j\right] \left(u-u\left(t_{n+1}\right)\right), |\De_j e^{n+1}|^{p-2} \De_j e^{n+1} \Rg \,dt' \\
\leq &\tau \left(  \|u\left( t_n\right)\|_{L^\infty}\|\nabla \De_j  \left(u-u\left(t_{n+1}\right)\right) \|_{L^\infty_tL^p} + \|\left[ u\left(t_n\right)\cdot \nabla, \LP \De_j\right] \left(u-u\left(t_{n+1}\right)\right)\|_{L^\infty_tL^p}\right) \|\De_j e^{n+1}\|^{p-1}_{L^p}\\
\leq& C \tau^2  2^j \|u_0\|_{B^1_{\infty,1}} \left( \|\De_j \left( u\cdot \nabla u\right)\|_{L^\infty_tL^p} + \nu \|\De_j \De u\|_{L^\infty_tL^p}\right)\|\De_j e^{n+1}\|^{p-1}_{L^p}\\
&+C\tau \|\left[ u\left(t_n\right)\cdot \nabla, \LP \De_j\right] \left(u-u\left(t_{n+1}\right)\right)\|_{L^\infty_tL^p} \|\De_j e^{n+1}\|^{p-1}_{L^p}.
\end{aligned}
\end{equation}
Similarly, we infer that
\begin{align}\label{I33}
    I_{3,3}&\leq \tau \|\De_j \left( e^n\cdot \nabla u\left(t_{n+1}\right)\right)\|_{L^p}   \|\De_j e^{n+1}\|^{p-1}_{L^p} .
\end{align}
Lastly, by the divergence free condition it follows that
\begin{equation}\label{I34}
\begin{aligned}
    I_{3, 4}=&\int^{t_{n+1}}_{t_n}   \Lg   \left (  u^n \cdot \nabla \LP \De_j\left (u\left(t_{n
    +1}\right)-u^{n+1}\right)\right),|\De_j e^{n+1}|^{p-2} \De_j e^{n+1}\Rg \, dt'\\
    &-\int^{t_{n+1}}_{t_n}   \Lg   \left [  u^n \cdot \nabla,  \LP \De_j\left (u\left(t_{n
    +1}\right)-u^{n+1}\right)\right],|\De_j e^{n+1}|^{p-2} \De_j e^{n+1}\Rg \, dt'\\
    =&-\int^{t_{n+1}}_{t_n}   \Lg   \left [  u^n \cdot \nabla,  \LP \De_j\left (u\left(t_{n
    +1}\right)-u^{n+1}\right)\right],|\De_j e^{n+1}|^{p-2} \De_j e^{n+1}\Rg\, dt'\\
    \leq&\tau \| \left [  u^n \cdot \nabla,  \LP \De_j e^{n+1}\right] \|_{L^p}\|\De_j e^{n+1}\|^{p-1}_{L^p}.
\end{aligned}
\end{equation}
Therefore collecting all the estimates \eqref{I31}-\eqref{I34} we can conclude that
\begin{equation}\label{I3}
    \begin{aligned}
        I_3 \leq &C \tau^2 2^j \left(\nu \|u_0\|_{B^2_{\infty,1}} +\|u_0\|^2_{B^2_{\infty,1}} \right)\|\De_j u\|_{L^\infty_tL^p} \|\De_j e^{n+1}\|^{p-1}_{L^p}\\
    &+ C\tau \|\left[ \left( u-u\left(t_n\right)\right)\cdot \nabla, \LP \De_j \right]u\|_{L^\infty_tL^p}\|\De_j e^{n+1}\|^{p-1}_{L^p} \\
    &+ \tau^2  2^j \|u_0\|_{B^1_{\infty,1}} \left( \|\De_j \left( u\cdot \nabla u\right)\|_{L^\infty_tL^p} + \nu \|\De_j \De u\|_{L^\infty_tL^p}\right)\|\De_j e^{n+1}\|^{p-1}_{L^p}\\
&+C\tau \|\left[ u\left(t_n\right)\cdot \nabla, \LP \De_j\right] \left(u-u\left(t_{n+1}\right)\right)\|_{L^\infty_tL^p} \|\De_j e^{n+1}\|^{p-1}_{L^p}\\
&+\tau \|\De_j \left( e^n\cdot \nabla u\left(t_{n+1}\right)\right)\|_{L^p}   \|\De_j e^{n+1}\|^{p-1}_{L^p}\\
&+C\tau \| \left [  u^n \cdot \nabla,  \LP \De_j e^{n+1} \right] \|_{L^p}\|\De_j e^{n+1}\|^{p-1}_{L^p}.
 \end{aligned}
\end{equation}
Based on the estimates \eqref{4.2}, \eqref{I1}, \eqref{I2} and\eqref{I3} above we finally conclude (by dividing both sides by $\|\De_j e^{n+1}\|_{L^p}^{p-1}$) that
\begin{equation}
\begin{aligned}
    \|\De_j e^{n+1}\|_{L^p} \leq& \|\De_j e^n\|_{L^p}+\nu \tau^2  2^{2j}  \left( \|\De_j \left( u\cdot \nabla u\right)\|_{L^\infty_tL^p} + \nu \|\De_j \De u\|_{L^\infty_tL^p}\right)\nonumber\\
    &+ \tau^2  2^j \left(\nu \|u_0\|_{B^2_{\infty,1}} +\|u_0\|^2_{B^2_{\infty,1}} \right)\| \De_j u\|_{L^\infty_tL^p} \nonumber\\
    &+ C\tau \|\left[ \left( u-u\left(t_n\right)\right)\cdot \nabla, \LP \De_j \right]u\|_{L^\infty_tL^p} \nonumber\\
    &+ \tau^2  2^j \|u_0\|_{B^1_{\infty,1}} \left( \|\De_j \left( u\cdot \nabla u\right)\|_{L^\infty_tL^p} + \nu \|\De_j \De u\|_{L^\infty_tL^p}\right)\nonumber\\
&+C\tau \|\left[ u\left(t_n\right)\cdot \nabla, \LP \De_j\right] \left(u-u\left(t_{n+1}\right)\right)\|_{L^\infty_tL^p} \nonumber\\
&+\tau \|\De_j \left( e^n\cdot \nabla u\left(t_{n+1}\right)\right)\|_{L^p}\nonumber\\
&+C\tau \| \left [  u^n \cdot \nabla,  \LP \De_j e^{n+1} \right] \|_{L^p}.
\end{aligned}
\end{equation}
Using Lemma~\ref{PL} and the uniform estimate \eqref{1.6}, we have
\begin{equation}\label{4.12}
    \| e^n\cdot \nabla u\left(t_{n+1}\right)\|_{B^0_{\infty,1}} \lesssim \|e^n\|_{B^0_{\infty,1}}\|u\|_{L^\infty_t B^1_{\infty,1}}\lesssim\|e^n\|_{B^0_{\infty,1}}\|u_0\|_{B^1_{\infty,1}}.
\end{equation}
On the other hand, by the commutator estimate Lemma~\ref{CI} and the stability estimate \eqref{1.6} we have
\begin{equation}\label{4.13}
     \sum_j \| \left [  u^n \cdot \nabla,  \LP \De_j e^{n+1} \right] \|_{L^\infty}\lesssim \|e^{n+1}\|_{B^0_{\infty,1}}\|u^n\|_{B^1_{\infty,1}}\lesssim\|e^{n+1}\|_{B^0_{\infty,1}}\|u_0\|_{B^1_{\infty,1}}.
\end{equation}
As a result, by \eqref{4.12}-\eqref{4.13}, letting $p\to\infty$ and summing $j$, we know
\begin{align*}
    \|e^{n+1}\|_{B^0_{\infty,1}} \leq &\|e^n\|_{B^0_{\infty,1}} +\nu\tau^2 \left(\| u_0\|_{B^1_{\infty,1}}\| u_0\|_{B^3_{\infty,1}} +\| u_0\|^2_{B^2_{\infty,1}}\right) + \nu^2\tau^2\| u_0\|_{B^4_{\infty,1}}\nonumber\\
    &+\tau^2   \left(\nu \|u_0\|_{B^2_{\infty,1}} +\|u_0\|^2_{B^2_{\infty,1}} \right)\| u_0\|_{B^1_{\infty,1}}\nonumber\\
&+ C\tau\|u-u\left(t_n\right)\|_{L^\infty_t B^1_{\infty,1}}\|u_0\|_{B^1_{\infty,1}}\nonumber\\
    &+ C\tau^2 \|u_0\|_{B^1_{\infty,1}}\left( \|u_0\|_{B^1_{\infty,1}}\|u_0\|_{B^2_{\infty,1}}+\|u_0\|^2_{B^1_{\infty,1}}+\nu\|u_0\|_{B^3_{\infty,1}} \right)\nonumber\\
    &+C\tau \|u_0\|_{B^1_{\infty,1}}\|u-u\left(t_{n+1}\right)\|_{L^\infty_t B^1_{\infty,1}}\nonumber\\
    &+C\tau\|e^n\|_{B^0_{\infty,1}}\|u_0\|_{B^1_{\infty,1}}\nonumber\\
    &+C\tau\|e^{n+1}\|_{B^0_{\infty,1}}\|u_0\|_{B^1_{\infty,1}}\nonumber
\end{align*}
   By introducing the time integral of $u-u(t_n)$ and $u-u(t_{n+1})$ we get
   \begin{align*}
     \|e^{n+1}\|_{B^0_{\infty,1}}\leq& \|e^n\|_{B^0_{\infty,1}} +\nu\tau^2 \left(\| u_0\|_{B^1_{\infty,1}}\| u_0\|_{B^3_{\infty,1}} +\| u_0\|^2_{B^2_{\infty,1}}\right) + \nu^2\tau^2\| u_0\|_{B^4_{\infty,1}}\nonumber\\
    &+\tau^2   \left(\nu \|u_0\|_{B^2_{\infty,1}} +\|u_0\|^2_{B^2_{\infty,1}} \right)\| u_0\|_{B^1_{\infty,1}}\nonumber\\
&+ C\tau^2\left(\|u_0\|_{B^1_{\infty,1}}\|u_0\|_{B^2_{\infty,1}}+\|u_0\|^2_{B^1_{\infty,1}} +\nu\|u_0\|_{B^3_{\infty,1}}\right)\|u_0\|_{B^1_{\infty,1}}\nonumber\\
    &+ C\tau^2 \|u_0\|_{B^1_{\infty,1}}\left( \|u_0\|_{B^1_{\infty,1}}\|u_0\|_{B^2_{\infty,1}}+\|u_0\|^2_{B^1_{\infty,1}}+\nu\|u_0\|_{B^3_{\infty,1}} \right)\nonumber\\
    &+C\tau\|e^n\|_{B^0_{\infty,1}}\|u_0\|_{B^1_{\infty,1}}\nonumber\\
&+C\tau\|e^{n+1}\|_{B^0_{\infty,1}}\|u_0\|_{B^1_{\infty,1}}\nonumber.
\end{align*}
Therefore if $u_0\in B^{4}_{\infty,1}$, this implies that
\begin{equation}\label{4.14}
\begin{aligned}
    \|u^{n+1}-u\left(t_{n+1}\right)\|_{B^0_{\infty,1}} \leq \frac{1+C^* \tau}{1-C^* \tau}  \|u^{n}-u\left(t_{n}\right)\|_{B^0_{\infty,1}} + \frac{C^* (1+\nu^2) \tau^2}{1-C^* \tau},
\end{aligned}
\end{equation}
for some $C^*$  depending on $\|u_0\|_{B^4_{\infty,1}}$. Applying the discrete Gr\"onwall's inequality by iterating the inequality \eqref{4.14} we get the desired error estimate:
\begin{equation}
     \sup_n \|u^{n}-u\left(t_{n}\right)\|_{B^0_{\infty,1}}\le C\tau
\end{equation}
for some $C>0$, which then concludes \eqref{1.8a}. The higher order error estimate \eqref{1.8b} is a direct corollary and we then omit the details.

\section{$B^0_{\infty,2}$ error estimates}\label{sec:Besov2}

Similar to the $B^0_{\infty,1}$ error estimates, we again perform the $L^p$-estimates for Littlewood-Paley projected $\De_j e^{n+1}:=\De_j u^{n+1}-\De_j u\left(t_{n+1}\right)$. We multiply \eqref{4.1} by $|\De_j e^{n+1}|^{p-2}\De_j e^{n+1}$ and integrate to obtain that 
\begin{equation}\label{5.1}
\begin{aligned}
    &\Lg \De_j e^{n+1}, |\De_j e^{n+1}|^{p-2} \De_j e^{n+1}  \Rg \\&= \Lg \De_j e^{n},|\De_j e^{n+1}|^{p-2} \De_j e^{n+1}\Rg\\
    &\quad + \nu \int^{t_{n+1}}_{t_n} \Lg \De \De_j u^{n+1}- \De \De_j u , |\De_j e^{n+1}|^{p-2} \De_j e^{n+1} \Rg dt' \\
    &\quad+  \int^{t_{n+1}}_{t_n} \Lg \LP \De_j (u \cdot \na u - u^{n} \cdot \na u^{n+1}), |\De_j e^{n+1}|^{p-2} \De_j e^{n+1} \Rg dt'\\
    &:= I'_1 + I'_2 + I'_3.
    \end{aligned}
\end{equation}
Here, the estimates of $I^{'}_1$ is same as the $I_1$ and therefore the following holds:
\begin{align}\label{I1'}
    I'_1 \leq \|\De_j e^{n+1}\|^{p-1}_{L^p}\|\De_j e^{n}\|_{L^p}.
\end{align}
For $I'_2$, we have
\begin{align*}
    I'_2 &=\nu \int^{t_{n+1}}_{t_n} \Lg \De\De_j u^{n+1}- \De\De_j u\left(t_{n+1}\right),  |\De_j e^{n+1}|^{p-2} \De_j e^{n+1} \Rg\, dt'\nonumber\\
    &\quad+ \nu \int^{t_{n+1}}_{t_n} \Lg \De\De_j u\left(t_{n+1}\right)- \De\De_j u, |\De_j e^{n+1}|^{p-2} \De_j e^{n+1} \Rg\, dt' \nonumber\\
    &\coloneqq I'_{2,1}+I'_{2,2}.
\end{align*}
The estimates of $I'_{2,1}$ is similar to $I_{2,1}$. Indeed we have
\begin{equation}\label{I21'}
    \begin{aligned}
    I'_{2,1}\coloneqq & \nu \int^{t_{n+1}}_{t_n} \Lg \De\De_j u^{n+1}- \De\De_j u\left(t_{n+1}\right), |\De_j e^{n+1}|^{p-2} \De_j e^{n+1} \Rg \, dt'\\
    =& \nu\tau \Lg \De\De_j e^{n+1} , |\De_j e^{n+1}|^{p-2} \De_j e^{n+1} \Rg\\
    \le&-\nu\tau2^{2j}\|\De_j e^{n+1}\|^p_{L^p}.
\end{aligned}
\end{equation}
Unlike the previous $B^0_{\infty,1}$ error case, the $-\nu\tau2^{2j}\|\De_j e^{n+1}\|^p_{L^p}$ term will play an important role in reducing the regularity. To continue, for $I^{'}_{2,2}$ we get
\begin{equation}\label{I22'}
\begin{aligned}
     I'_{2,2} \coloneqq& -\nu 2^{2j} \int^{t_{n+1}}_{t_n} \Lg \De_j u\left(t_{n+1}\right)- \De_j u, |\De_j e^{n+1}|^{p-2} \De_j e^{n+1} \Rg\, dt'\\
     \le& \nu 2^{2j} \|\De_j e^{n+1}\|^{p-1}_{L^p}  \int^{t_{n+1}}_{t_n} \|\De_j u\left(t_{n+1}\right)- \De_j u\|_{L^p} \, dt'\\
     \le& \nu 2^{2j} \|\De_j e^{n+1}\|^{p-1}_{L^p}  \int^{t_{n+1}}_{t_n} \|\int^{t_{n+1}}_{t'} \partial_t \De_j u \, ds \|_{L^p} \, dt'\\
      \le& \nu 2^{2j} \|\De_j e^{n+1}\|^{p-1}_{L^p}   \int^{t_{n+1}}_{t_n} \int^{t_{n+1}}_{t'} \| \partial_t \De_j u  \|_{L^p} \, ds\, dt' \\
     \le& \frac{1}{2}\nu\tau2^{2j}\|\De_j e^{n+1}\|^p_{L^p}\\
     &+\frac{1}{2} \nu \tau^3 2^{2j}  \left( \|\De_j \left( u\cdot \nabla u\right)\|^2_{L^\infty_tL^p} + \nu^2 \|\De_j \De u\|^2_{L^\infty_tL^p}\right) \|\De_j e^{n+1}\|^{p-2}_{L^p}.\\
\end{aligned}
\end{equation}
It then follows by \eqref{I21'} and \eqref{I22'} that
\begin{align}\label{I2'}
     I'_{2} \leq\frac{1}{2} \nu \tau^3 2^{2j}  \left( \|\De_j \left( u\cdot \nabla u\right)\|^2_{L^\infty_tL^p} + \nu^2 \|\De_j \De u\|^2_{L^\infty_tL^p}\right) \|\De_j e^{n+1}\|^{p-2}_{L^p}. 
\end{align}
For $I'_3$, we will estimate similarly as before:
\begin{equation}
\begin{aligned}
      I'_3&= \int^{t_{n+1}}_{t_n} \Lg  \LP \De_j \left ( \left( u-u\left(t_n\right)\right)\cdot \nabla u\right),  |\De_j e^{n+1}|^{p-2} \De_j e^{n+1}\Rg dt'\nonumber\\
    &\quad+\int^{t_{n+1}}_{t_n} \Lg  \LP \De_j \left ( u\left(t_n\right)\cdot \nabla \left(u-u\left(t_{n+1}\right)\right)\right), |\De_j e^{n+1}|^{p-2} \De_j e^{n+1} \Rg dt' \nonumber\\
    &\quad+ \int^{t_{n+1}}_{t_n} \Lg  \LP \De_j  \left ( \left( u\left(t_n\right)-u^n\right)\cdot \nabla u\left(t_{n
    +1}\right)\right), |\De_j e^{n+1}|^{p-2} \De_j e^{n+1} \Rg dt'\nonumber\\
    &\quad+  \int^{t_{n+1}}_{t_n}   \Lg  \LP \De_j  \left (  u^n \cdot \nabla \left (u\left(t_{n
    +1}\right)-u^{n+1}\right)\right),|\De_j e^{n+1}|^{p-2} \De_j e^{n+1}\Rg dt'\nonumber\\
    &:= I'_{3,1} + I'_{3,2} + I'_{3,3} + I'_{3,4}
\end{aligned}
\end{equation}
We first consider the term $I'_{3,1}$. Unlike the $B^0_{\infty,1}$ case, we shall only estimate by H\"older's inequality:
\begin{equation}\label{I31'}
I'_{3,1} \leq \tau\|\De_j \left ( \left( u-u\left(t_n\right)\right)\cdot \nabla u\right) \|_{L^\infty_tL^p}\| \De_j e^{n+1}\|^{p-1}_{L^p}.    
\end{equation}
Similarly we have
\begin{equation}\label{I32'}
    \begin{aligned}
I'_{3,2} 
\leq&  \tau^2  2^j \|u_0\|_{B^1_{\infty,1}} \left( \|\De_j \left( u\cdot \nabla u\right)\|_{L^\infty_tL^p} + \nu \|\De_j \De u\|_{L^\infty_tL^p}\right)\|\De_j e^{n+1}\|^{p-1}_{L^p}\\
&+C\tau \|\left[ u\left(t_n\right)\cdot \nabla, \LP \De_j\right] \left(u-u\left(t_{n+1}\right)\right)\|_{L^\infty_tL^p} \|\De_j e^{n+1}\|^{p-1}_{L^p}
\end{aligned}
\end{equation}
and
\begin{align}\label{I33'}
    I'_{3,3}&\leq \tau \|\De_j \left( e^n\cdot \nabla u\left(t_{n+1}\right)\right)\|_{L^p}   \|\De_j e^{n+1}\|^{p-1}_{L^p} .
\end{align}
Lastly, again by the divergence free condition it follows that
\begin{equation}\label{I34'}
\begin{aligned}
    I'_{3, 4}
    =&-\int^{t_{n+1}}_{t_n}   \Lg   \left [  u^n \cdot \nabla,  \LP \De_j\left (u\left(t_{n
    +1}\right)-u^{n+1}\right)\right],|\De_j e^{n+1}|^{p-2} \De_j e^{n+1}\Rg\, dt'\\
    \leq&\tau \| \left [  u^n \cdot \nabla,  \LP \De_j e^{n+1}\right] \|_{L^p}\|\De_j e^{n+1}\|^{p-1}_{L^p}.
\end{aligned}
\end{equation}
Therefore by collecting the estimates \eqref{I1'}, \eqref{I2'} and \eqref{I31'}-\eqref{I34'} we have
\begin{equation}\label{5.10}
\begin{aligned}
    \|\De_j e^{n+1}\|^2_{L^p} \leq &\|\De_j e^n\|_{L^p}\|\De_j e^{n+1}\|_{L^p} +\frac{1}{2} \nu \tau^3 2^{2j}  \left( \|\De_j \left( u\cdot \nabla u\right)\|^2_{L^\infty_tL^p} + \nu^2 \|\De_j \De u\|^2_{L^\infty_tL^p}\right)\|\De_j e^{n+1}\|_{L^p}\\
    &+ C\tau \|\De_j\left ( \left( u-u\left(t_n\right)\right)\cdot \nabla u\right)  \|_{L^\infty_tL^p} \|\De_j e^{n+1}\|_{L^p}\\
     &+ \tau^2  2^j \|u_0\|_{B^1_{\infty,1}} \left( \|\De_j \left( u\cdot \nabla u\right)\|_{L^\infty_tL^p} + \nu \|\De_j \De u\|_{L^\infty_tL^p}\right)\|\De_j e^{n+1}\|_{L^p}\\
&+C\tau \|\left[ u\left(t_n\right)\cdot \nabla, \LP \De_j\right] \left(u-u\left(t_{n+1}\right)\right)\|_{L^\infty_tL^p}\|\De_j e^{n+1}\|_{L^p} \\
&+\tau \|\De_j \left( e^n\cdot \nabla u\left(t_{n+1}\right)\right)\|_{L^p}\|\De_j e^{n+1}\|_{L^p}\\
&+C\tau \| \left [  u^n \cdot \nabla,  \LP \De_j e^{n+1} \right] \|_{L^p}\|\De_j e^{n+1}\|_{L^p}\\
\end{aligned}    
\end{equation}
It is worth pointing our there that the inequality above was derived by dividing both sides by $\|\De_j e^{n+1}\|_{L^p}^{p-2}$. This is the main difference in the $B^0_{\infty,2}$ setting as in the $B^0_{\infty,1}$ case. Note that by Cauchy-Schwarz inequality and Lemma~\ref{PL2}, we have the following estimates hold for any $t\in(t_n,t_{n+1})$:
\begin{equation}\label{5.11}
\begin{aligned}
   &\sum_j \|\De_j\left ( \left( u-u\left(t_n\right)\right)\cdot \nabla u\right)  \|_{L^\infty} \|\De_j e^{n+1}\|_{L^\infty}\\
   \lesssim&\left\|(\|\De_j\left ( \left( u-u\left(t_n\right)\right)\cdot \nabla u\right)  \|_{L^\infty})_j\right\|_{l^2}\left\|(\|\De_j e^{n+1}\|_{L^\infty})_j\right\|_{l^2}\\
   \lesssim& \| (u-u(t_n))\cdot \na u\|_{B^0_{\infty,2}}\|e^{n+1}\|_{B^0_{\infty,2}}\\
   \lesssim& \|u-u(t_n)\|_{B^0_{\infty,2}}\|u\|_{B^2_{\infty,2}}\|e^{n+1}\|_{B^0_{\infty,2}},
\end{aligned}
\end{equation}
and
\begin{equation}\label{5.12}
\begin{aligned}
   &\sum_j \|\De_j \left( e^n\cdot \nabla u\left(t_{n+1}\right)\right)\|_{L^\infty}\|\De_j e^{n+1}\|_{L^\infty}\\
   \lesssim& \| e^n\cdot \na u(t_{n+1})\|_{B^0_{\infty,2}}\|e^{n+1}\|_{B^0_{\infty,2}}\\
   \lesssim& \|e^n\|_{B^0_{\infty,2}}\|u(t_{n+1})\|_{B^2_{\infty,2}}\|e^{n+1}\|_{B^0_{\infty,2}}.
\end{aligned}
\end{equation}
By Lemma~\ref{CI} and Lemma~\ref{embed} we have
\begin{equation}\label{5.13}
    \begin{aligned}
        &\sum_j\|\left[ u\left(t_n\right)\cdot \nabla, \LP \De_j\right] \left(u-u\left(t_{n+1}\right)\right)\|_{L^\infty}\|\De_j e^{n+1}\|_{L^\infty}\\
        \lesssim& \left\| \left(\|\left[ u\left(t_n\right)\cdot \nabla, \LP \De_j\right] \left(u-u\left(t_{n+1}\right)\right)\|_{L^\infty}\right)_j\right\|_{l^2}\|e^{n+1}\|_{B^0_{\infty,2}}\\
        \lesssim& \|u(t_n)\|_{B^1_{\infty,1}}\|u-u\left(t_{n+1}\right)\|_{B^0_{\infty,2}}\|e^{n+1}\|_{B^0_{\infty,2}}\\
        \lesssim& \|u(t_n)\|_{B^2_{\infty,2}}\|u-u\left(t_{n+1}\right)\|_{B^0_{\infty,2}}\|e^{n+1}\|_{B^0_{\infty,2}}.
    \end{aligned}
\end{equation}
Moreover, we have
\begin{equation}
    \begin{aligned}
        \sum_j \| \left [  u^n \cdot \nabla,  \LP \De_j e^{n+1} \right] \|_{L^\infty}\|\De_j e^{n+1}\|_{L^\infty}
\lesssim\|u^n\|_{B^2_{\infty,2}}\|e^{n+1}\|^2_{B^0_{\infty,2}}.
    \end{aligned}
\end{equation}
Now by letting $p\to\infty$, summing $j$, applying the uniform estimates\eqref{1.7}  and stability estimates Theorem~\ref{Thm1.1} and collecting all of the estimates \eqref{5.10}-\eqref{5.13} we have
\begin{align*}
    \|e^{n+1}\|^2_{B^0_{\infty,2}}\leq& \|e^{n}\|_{B^0_{\infty,2}} \|e^{n+1}\|_{B^0_{\infty,2}}+ C\nu\tau^3\left( \|u_0\|^4_{B^2_{\infty,2}} + \nu^2 \|u_0\|^2_{B^3_{\infty,2}}\right)\|e^{n+1}\|_{B^0_{\infty,2}}\nonumber\\
    &+C\tau\|u-u\left(t_n\right)\|_{B^0_{\infty,2}}\|u_0\|_{B^2_{\infty,2}}\|e^{n+1}\|_{B^0_{\infty,2}}\nonumber\\
    &+\tau^2 \|u_0\|_{B^2_{\infty,2}}\left(  \|u_0\|^2_{B^2_{\infty,2}} + \nu \|u_0\|_{B^3_{\infty,2}}          \right)\|e^{n+1}\|_{B^0_{\infty,2}}\nonumber\\
    &+C\tau\|u_0\|_{B^2_{\infty,2}}\|u-u\left(t_{n+1}\right)\|_{B^0_{\infty,2}}\|e^{n+1}\|_{B^0_{\infty,2}}\nonumber\\
&+C\tau\|e^n\|_{B^0_{\infty,2}}\|u_0\|_{B^2_{\infty,2}}\|e^{n+1}\|_{B^0_{\infty,2}} \nonumber\\
    &+C\tau \|u_0\|_{B^2_{\infty,2}}\|e^{n+1}\|^2_{B^0_{\infty,2}}\nonumber.
    \end{align*}
Again by introducing the time integral of $u-u(t_n)$ and $u-u(t_{n+1})$ we get
    \begin{align*}
   \|e^{n+1}\|^2_{B^0_{\infty,2}}\leq &\|e^{n}\|_{B^0_{\infty,2}} \|e^{n+1}\|_{B^0_{\infty,2}}+ C\nu\tau^3\left( \|u_0\|^4_{B^2_{\infty,2}} + \nu^2 \|u_0\|^2_{B^3_{\infty,2}}\right)\|e^{n+1}\|_{B^0_{\infty,2}}\nonumber\\
    &+C\tau^2\left( \|u_0\|^2_{B^1_{\infty,2}} +\nu \|u_0\|_{B^2_{\infty,2}}\right)\|u_0\|_{B^2_{\infty,2}}\|e^{n+1}\|_{B^0_{\infty,2}}\nonumber\\
    &+\tau^2 \|u_0\|_{B^2_{\infty,2}}\left(  \|u_0\|^2_{B^2_{\infty,2}} + \nu \|u_0\|_{B^3_{\infty,2}}          \right)\|e^{n+1}\|_{B^0_{\infty,2}}\nonumber\\
&+C\tau^2\|u_0\|_{B^2_{\infty,2}}\left(\|u_0\|^2_{B^1_{\infty,2}}+\nu\|u_0\|_{B^2_{\infty,2}}\right)\|e^{n+1}\|_{B^0_{\infty,2}}\\   
&+C\tau\|e^n\|_{B^0_{\infty,2}}\|u_0\|_{B^2_{\infty,2}}\|e^{n+1}\|_{B^0_{\infty,2}} \nonumber\\
    &+C\tau \|u_0\|_{B^2_{\infty,2}}\|e^{n+1}\|^2_{B^0_{\infty,2}}\nonumber.
\end{align*}
By Cauchy-Schwarz inequality one can get
\begin{align*}
   \|e^{n+1}\|^2_{B^0_{\infty,2}}\leq \|e^{n}\|^2_{B^0_{\infty,2}}+C\tau (\|e^{n+1}\|^2_{B^0_{\infty,2}}+\|e^{n}\|^2_{B^0_{\infty,2}})+C\tau^3(1+\nu+\nu\tau^2+\nu^3\tau^2),
\end{align*}
which implies that
\begin{equation}\label{5.x}
\begin{aligned}
    \|u^{n+1}-u\left(t_{n+1}\right)\|^2_{B^0_{\infty,2}} \leq \frac{1+C^* \tau}{1-C^* \tau}  \|u^{n}-u\left(t_{n}\right)\|^2_{B^0_{\infty,2}} + \frac{C^* (1+\nu+\nu\tau^2+\nu^3\tau^2) \tau^3}{1-C^* \tau},
\end{aligned}
\end{equation}
where $C^*$  depends on $\|u_0\|_{B^3_{\infty,2}}$. Iterating the inequality \eqref{5.x} we arrive at 
\begin{align*}
   \sup_n \|u^{n}-u\left(t_{n}\right)\|^2_{B^0_{\infty,2}}\leq C
    \tau^2,
\end{align*}
for some $C>0$. This then concludes \eqref{1.9a}. The higher order error estimate \eqref{1.9b} is a direct corollary and we also leave the details to the readers.

\section{Extension to the full discretization}\label{sec:full}
In this section we briefly discuss how to extend the proposed semi-implicit scheme to full discretization. More precisely speaking we give a brief instruction on how to apply finite element methods or Fourier spectral methods.

\subsection{Finite element method}
Our semi-implicit method can be extended with standard finite element methods. For example, a pair of finite element subspace of mesh size $h$, 
$V_h\times Q_h \subset H_0^1\times L_0^2$ with the following properties are able to fulfill the task. Here $H_0^1$ and $L^2_0$ contain $H^1$ and $L^2$ functions of zero mean respectively. Indeed we refer the readers to \cite{LMU22} and \cite{LQY22} for more a detailed set-up. 

\begin{enumerate}

\item[(P1)] There exists a linear projection operator $\Pi_h: H^1_0\rightarrow V_h$ such that 
\begin{enumerate}
\item[(i)]
$\nabla\cdot \Pi_hv = P_{Q_h} \nabla\cdot v$ for $v\in H^1_0$, where $P_{Q_h}:L^2_0\rightarrow Q_h $ denotes the $L^2$-orthogonal projection. 

\item[(ii)]
The following approximation property holds for $v\in H^1_0\cap H^m$:
\begin{align}\label{Fortin-Error} 
\|v-\Pi_hv\|_{H^{s}} 
\le Ch^{m-s}\|v\|_{H^{m}},\quad 0\le s\le 1,\quad 1\le m\le 2.
\end{align}

\end{enumerate}

\item[(P2)] $\nabla\cdot v_h\in Q_h $ for $v_h\in V_h$. 

\end{enumerate}
The two properties above guarantee the inf-sup condition for the pair $V_h\times Q_h $, i.e., 
\begin{align}\label{inf-sup}
\|q_h\|_{L^2}
\le 
\sup_{v_h\in V_h \backslash \{0\} }
\frac{C(\nabla\cdot v_h,q_h)}{\|v_h\|_{H^1}} \quad\forall\,q_h\in Q_h .
\end{align}
%



It is known that the discrete divergence-free subspace of $V_h$ coincides with its pointwise divergence-free subspace, i.e., 
\begin{align}\label{def-Xh}
X_h:=\{v_h\in V_h: (\nabla\cdot v_h, q_h) =0\,\,\, \forall\, q_h\in Q_h \} 
= \{v_h\in V_h: \nabla\cdot v_h=0\}  .
\end{align}
Therefore,
$$V_h\subset H^1_0\quad\mbox{and}\quad X_h\subset \dot H^1_0\subset \dot L^2 = X  ,$$ 
where $\dot H^1_0$ and $\dot L^2 $ are divergence-free subspaces of $H^1_0$ and $L^2$. Then it suffices to find $u_h^n\in X_h$, $n=1,\dots,N$, such that 
\begin{align}\label{fully-FEM-Euler}
\left\{\begin{aligned}
\bigg(\frac{ u_h^{n+1} - u_h^{n} }{\tau}, v_h\bigg)
+ \nu(\nabla u_h^{n+1}, \nabla  v_h)
+ (u_h^{n}\cdot\nabla u_h^{n+1}, v_h) &= 0 \quad\forall\, v_h\in X_h,\,\,\,  \\        
u_h^0&=P_{X_h}u^0,
\end{aligned}\right.
\end{align}
where $P_{X_h}:\dot L^2\rightarrow X_h$ denotes the $L^2$-orthogonal projection onto $X_h$.

\begin{thm}\label{THM:FEM-Euler}
Let $u^0\in \dot L^2$ and assume that the finite element space $X_h\times Q_h$ has properties {\rm(P1)--(P2).} Assume that the solution of the NS problem \eqref{NSE} has the following regularity: 
\begin{align}\label{reg-u-fd}
u\in C([0,T];L^2)\cap L^\infty(0,T;H^3), 
\end{align} 
then the fully discrete solution given by \eqref{fully-FEM-Euler} has the following error bound: 
\begin{align}\label{Fully_Error_bound}
\|u_h^n-u(n\tau) \|_{L^2} 
\le 
C ( \tau +  h ),
\end{align}
where the constant $C$ depends only on $u^0$ and $T$. 
\end{thm}

The proof can be found in \cite{LMU22}. It is worth pointing out that the assumption $u\in L^\infty(0,T;H^3)$ can be seen through an integration-by-parts technique as in \cite{CLW24} without introducing the smoothing effect of NS equations as in \cite{LMU22}. Moreover, due to the embeddings $B^4_{\infty,1}(\T^2)\hookrightarrow H^3(\T^2)$ and $B^3_{\infty,2}(\T^2)\hookrightarrow H^3(\T^2)$ one can still obtain the $L^2$-error bound \eqref{Fully_Error_bound} under the regularity assumptions 
\begin{align}\label{new_reg}
    u\in C([0,T];L^2)\cap L^\infty(0,T;B^4_{\infty,1})\quad \mbox{or} \quad u\in C([0,T];L^2)\cap L^\infty(0,T;B^3_{\infty,2})
\end{align}
as in Theorem~\ref{Thm1.2} and Theorem~\ref{Thm1.3}.

On the other hand, it seems very challenging to obtain a Besov error bounds such as $\|u_h^n-u(n\tau) \|_{B^0_{\infty,1}} $ or $\|u_h^n-u(n\tau) \|_{B^0_{\infty,2}} $ under the same assumptions in \eqref{new_reg} above. The main difficulty is lack of approximation property (P1) of the finite element projection in Besov spaces. Indeed there are many works in such area and we refer the readers to \cite{BG02,BDD04,Gan17,Guer04} for more discussion. More specifically speaking, $L^p$ approximation of the finite element space can be obtained as long as the function is in $B^{\alpha}_{p,p}$ from the physical point of view instead of frequency. However, it seems that one requires extra regularity to get $\|u_h^n-u(n\tau) \|_{B^0_{\infty,1}}$ error due to lack of the tools. Since we focus on a general Besov framework, we will not dig further in this direction in this paper.

\subsection{Fourier spectral Picard iteration method}

In this subsection we show that the semi-implicit scheme can be extended to the full discretization using a Fourier spectral Picard iteration method:
\begin{equation}\label{5.1}
    \begin{cases}
        &\frac{u^{(m+1)}-u^n}{\tau}+\LP \Pi_N(u^n\cdot \na u^{(m)})=\nu\De u^{(m+1)},\\
        &\div u^{(m+1)}=0,\\
        &u^{(0)}=u^n,\quad u^0=\Pi_N u_0,
    \end{cases}
\end{equation}
where $\tau$ is the time step. For $N\geq 2$, we introduce the space
$$X_N=\text{span}\left\{\cos(k\cdot x)\ ,\ \sin(k\cdot x):\ \ k=(k_1,k_2)\in\Z^2\ ,\ |k|_\infty=\max\{|k_1|,|k_2|\}\in[1, N] \right\} .$$
We define $\Pi_N$ to be the Fourier truncation operator $1\le|k|_\infty\leq N$. Here at each time step $n$ we implement an iteration of $u^{(m)}$ and eventually $u^{(m)}$ converges. We therefore define $\lim_m u^{(m)}=u^{n+1}$ up to a local tolerance. Such iteration \eqref{5.1} converges and we refer the readers to \cite{CLW24} for the details. The reason that we introduce such Picard iteration is to simplify the computation; indeed one can still compute \eqref{1.5} via Fourier spectral method with a convolution solver. The semi-implicit Fourier spectral (Picard iteration) scheme we consider is the following:
\begin{equation}\label{6.9}
\begin{cases}
&\frac{u^{n+1}-u^n}{\tau}+\LP\Pi_N(u^{n}\cdot \na u^{n+1}) = \nu \Delta u^{n+1},\\
&\div u^{n+1}=0,\\
&u^0=\Pi_N u_0.
\end{cases}
\end{equation}
We have the following Besov error estimate for the scheme \eqref{6.9}.

\begin{thm}\label{Thm6.1}Assume that $u(t,x) $ is the exact solution to \eqref{NSE}. 
\begin{description}
\item[(i) $B^0_{\infty,1}$-error]
Assume $ u\in C([0,T];L^2)\cap L^\infty(0,T;B^4_{\infty,1})$, then the following $B^0_{\infty,1}$-error estimate holds:
 \begin{equation}\label{6a}
   \sup_{n} \|u^{n} -u
   \left( n\tau\right)\|_{B^0_{\infty,1}}\leq C_1(\tau+N^{-3}).
 \end{equation}

\item[(ii) $B^0_{\infty,2}$-error]
Assume $ u\in C([0,T];L^2)\cap L^\infty(0,T;B^3_{\infty,2})$, then the following $B^0_{\infty,2}$-error estimate holds for any $\delta>0$:
 \begin{equation}\label{6b}
   \sup_{n} \|u^{n} -u
   \left( n\tau \right)\|_{B^0_{\infty,2}}\leq C_2(\tau+N^{-2+\delta}).
 \end{equation}
\end{description}
Here the constant $C_1>0$ above only depends on the initial condition $u_0$ while $C_2>0$ depends on $u_0$ and $\delta>0$. Both constants are independent of the viscosity $\nu$.

\end{thm}

The proof of Theorem~\ref{Thm6.1} is very similar to \cite{CLW24} and \cite{CLW25} and therefore we only sketch the proof here. First we note that $\Pi_Nu_0\in X_N$ and therefore by induction we have $u^n\in X_N\ ,\forall n\geq 0$. In addition we can further derive that $\div u^n=0$ for all $n$. We then discretize the continuous NS equation as follows:
\begin{equation}
\begin{aligned}
    u^{n+1}-u\left(t_{n+1}\right) &=u^n - \tau\Pi_N\LP(u^{n} \cdot \na u^{n+1}) + \tau \nu \De u^{n+1} -u\left(t_n\right) +\int^{t_{n+1}}_{t_n}   \LP(u \cdot \na u) - \nu \De u \, dt'\nonumber\\
    &=u^n-u\left(t_n\right)+ \nu \int^{t_{n+1}}_{t_n} \De u^{n+1}- \De u\, dt' +  \int^{t_{n+1}}_{t_n} \Pi_N\LP(u \cdot \na u - u^{n} \cdot \na u^{n+1})\, dt'\nonumber\\
    &\quad+ \int^{t_{n+1}}_{t_n} \left(\mathcal{I}-\Pi_N\right)\LP(u \cdot \na u)\, dt'.
\end{aligned}
\end{equation}
It is clear that applying the Little-wood Paley projection $\De_j$ we can derive that
\begin{equation}
\begin{aligned}
    &\De_j u^{n+1}-\De_j u\left(t_{n+1}\right) \\
    =&\De_j u^n-\De_j u\left(t_n\right)+ \nu \int^{t_{n+1}}_{t_n} \De\De_j u^{n+1}- \De \De_j u\, dt' +  \int^{t_{n+1}}_{t_n} \Pi_N\LP\De_j (u \cdot \na u - u^{n} \cdot \na u^{n+1})\, dt'\\
    &\quad+ \int^{t_{n+1}}_{t_n} \left(\mathcal{I}-\Pi_N\right)\LP\De_j(u \cdot \na u)\, dt'.
\end{aligned}
\end{equation}
Note that $\Pi_N\De_j=\De_j$ or $0$, therefore 
\begin{align*}
\widetilde{I}_3\coloneqq \int^{t_{n+1}}_{t_n} \Lg \LP \Pi_N\De_j (u \cdot \na u - u^{n} \cdot \na u^{n+1}), |\De_j e^{n+1}|^{p-2} \De_j e^{n+1} \Rg \, dt'    
\end{align*}
can be estimated exactly as $I_3$ or $I_3'$. Now the only extra term (compared to \eqref{4.1}) is $\sum_j\|B_{n}\|_{L^\infty}$ or $\sum_j\|B_{n}\|^2_{L^\infty}$, where
\begin{align*}
    B_n=\int^{t_{n+1}}_{t_n} \left(\mathcal{I}-\Pi_N\right)\LP\De_j(u \cdot \na u)\, dt'.
\end{align*}
Recall the definition of the frequency cut-off operator $\De_j$ back in \eqref{aeq2}. In short words, $\sum_j (\mathcal{I}-\Pi_N)\De_j\approx \sum_{2^j\ge N}\De_j$ in our setting. Therefore we have
\begin{align*}
  \sum_j\|B_n\|_{L^\infty} \lesssim\sum_{2^j\ge N}N^{-3} \|\na^3 B_n\|_{\infty}\lesssim \sum_{2^j\ge N}\tau N^{-3} \|\na^3 \De_j (u\cdot \na u)\|_{L^\infty_t L^\infty}.
    \end{align*}
Therefore we get by Lemma~\ref{PL1}
\begin{align*}
  \sum_j\|B_n\|_{L^\infty} \lesssim \tau N^{-3} \|u\|^2_{L^\infty_t B^4_{\infty,1}}.
    \end{align*}
The following error estimate then holds:
\begin{align*}
          \|e^{n+1}\|_{B^0_{\infty,1}} \leq \frac{1+C^* \tau}{1-C^* \tau}  \|e^{n}\|_{B^0_{\infty,1}} + \frac{C^* (1+\nu) \tau^2}{1-C^* \tau}+\frac{C^*\tau N^{-3}}{1-C^*\tau},
    \end{align*}
which proves \eqref{6a}. On the other hand, we note that 
\begin{align*}
  \sum_j\|B_n\|^2_{L^\infty} \lesssim&\sum_{2^j\ge N}N^{-4+2\delta} \|\Lg\na\Rg^{2-\de} B_n\|^2_{\infty}\\\lesssim&\tau^2N^{-4+2\de}\|u\cdot \na u\|^2_{L^\infty_tB^{2-\de}_{\infty,2}}\\
  \lesssim&\tau^2N^{-4+2\de}\|u\|^4_{L^\infty_tB^{3}_{\infty,2}}.
    \end{align*}
The error estimate \eqref{6b} then follows. {We see that the non-sharpness of $B^0_{\infty,2}$-error $N^{-2+\de}$ results from the fact that $B^0_{\infty,2}$ is not a multiplicative algebra, see Lemma~\ref{PL2} or Lemma~\ref{embed} and the discussion therein. }

\section{Numerical Experiments}\label{sec7}
In this section, we present several numerical examples to support our theoretical framework. We will demonstrate the error convergence and the vanishing viscosity rate in the following subsections.

\subsection{Convergence test}
In this subsection we present the $B^0_{\infty,1}$ and $ B^0_{\infty,2}$-error of the scheme \eqref{1.3}. More specifically, we will choose the following initial data $\vec{u_0}=(-0.5\sin(x)\cos(y),0.5\cos(x)\sin(y))$. This is suggested by \cite{CLW24}. Indeed we shall solve a forced incompressible NS system with explicit solution $\vec{u_e}$ given by the follows:
\begin{equation}\label{ue}
\vec{u_e}=(-0.5\exp(-t)\sin(x)\cos(y),0.5\exp(-t)\cos(x)\sin(y))
\end{equation}
with a forcing term $\vec{f_e}$ that can be computed explicitly. We then solve the scheme \eqref{1.5} with forcing:
{\begin{equation}\label{7.2}
\frac{u^{n+1}-u^n}{\tau} +\LP(u^n \cdot \na u^{n+1})=\nu\Delta u^{n+1}+f_e^{n+1}.
\end{equation}}
The $B^0_{\infty,1}$ and $B^0_{\infty,2}$-errors with different $\nu$-values at $T=2$ are given in the Table~\ref{table1}-\ref{table4} below and Figure~\ref{fig1}-\ref{fig4}. We can conclude that both errors are of order $O(\tau)$. 

\begin{table}[htb]
\begin{minipage}[b]{0.48\linewidth}
\centering
\begin{tabular}{c  c c c c} 
 \toprule 
 $\tau=0.01$ &  $B^0_{\infty,1}$-error & $B^0_{\infty,2}$-error \\ [0.5ex]
 \hline
 $\tau$  & 0.0018 & 0.0013 \\[0.5ex]
 
 $\tau/2$ &8.827e-04 &6.242e-04
 \\[0.5ex]
$\tau/4$  &4.401e-04&3.112e-04
 \\[0.5ex]
 $\tau/8$  &2.197e-04 &1.554e-04
 \\[0.5ex]
 $\tau/16$  & 1.098e-04&7.763e-05
 \\[0.5ex]
 $\tau/32$  &5.487e-05&3.880e-05\\
\bottomrule 
\end{tabular}
\captionof{table}{Errors with $\nu=1$.}\label{table1}
\end{minipage}
\hfill
\begin{minipage}[b]{0.48\linewidth}
\centering
\begin{tabular}{c  c c c c} 
 \toprule 
 $\tau=0.01$ & $B^0_{\infty,1}$-error & $B^0_{\infty,2}$-error \\ [0.5ex]
 \hline
 $\tau$   & 0.0020 & 0.0014 \\[0.5ex]
 
 $\tau/2$   &0.0010 & 7.100e-04
 \\[0.5ex]
$\tau/4$ &5.018e-04&3.548e-04
 \\[0.5ex]
 $\tau/8$   &2.508e-04&1.774e-04
 \\[0.5ex]
 $\tau/16$   & 1.254e-04&8.867e-05
 \\[0.5ex]
 $\tau/32$   &6.270e-05&4.433e-05\\
\bottomrule 
\end{tabular}
\captionof{table}{Errors with $\nu=0.1$.}\label{table2}
\end{minipage}
\end{table}

\begin{table}[htb]
\begin{minipage}[b]{0.48\linewidth}
\centering
\begin{tabular}{c  c c c c} 
 \toprule 
 $\tau=0.01$ & $B^0_{\infty,1}$-error & $B^0_{\infty,2}$-error \\ [0.5ex]
 \hline
 $\tau$  & 0.0040 & 0.0029 \\[0.5ex]
 
 $\tau/2$   &0.0020& 0.0014
 \\[0.5ex]
$\tau/4$ &0.0010 &7.147e-04
 \\[0.5ex]
 $\tau/8$   &5.054e-04&3.574e-04
 \\[0.5ex]
 $\tau/16$   & 2.527e-04& 1.787e-04
 \\[0.5ex]
 $\tau/32$   &1.263e-04&8.934e-05\\
\bottomrule %
\end{tabular}
\captionof{table}{Errors with $\nu=0.01$.}\label{table3} 
\end{minipage}
\hfill
\begin{minipage}[b]{0.48\linewidth}
\centering
\begin{tabular}{c  c c c c} 
 \toprule 
 $\tau=0.01$ & $B^0_{\infty,1}$-error & $B^0_{\infty,2}$-error \\ [0.5ex]
 \hline
 $\tau$ & 0.0086  & 0.0061  \\[0.5ex]
 
 $\tau/2$  &0.0043 & 0.0031 
 \\[0.5ex]
$\tau/4$ & 0.0022 & 0.0015
 \\[0.5ex]
 $\tau/8$  &0.0011 &7.642e-04 
 \\[0.5ex]
 $\tau/16$  &  5.404e-04&3.821e-04
 \\[0.5ex]
 $\tau/32$  &2.702e-04& 1.911e-04 \\
\bottomrule 
\end{tabular}
\captionof{table}{Errors with $\nu=0.00001$.}\label{table4}
\end{minipage}

\end{table}

\begin{figure}[htb]
\begin{minipage}[htb]{0.48\linewidth}
    \centering
\includegraphics[width=1\textwidth]{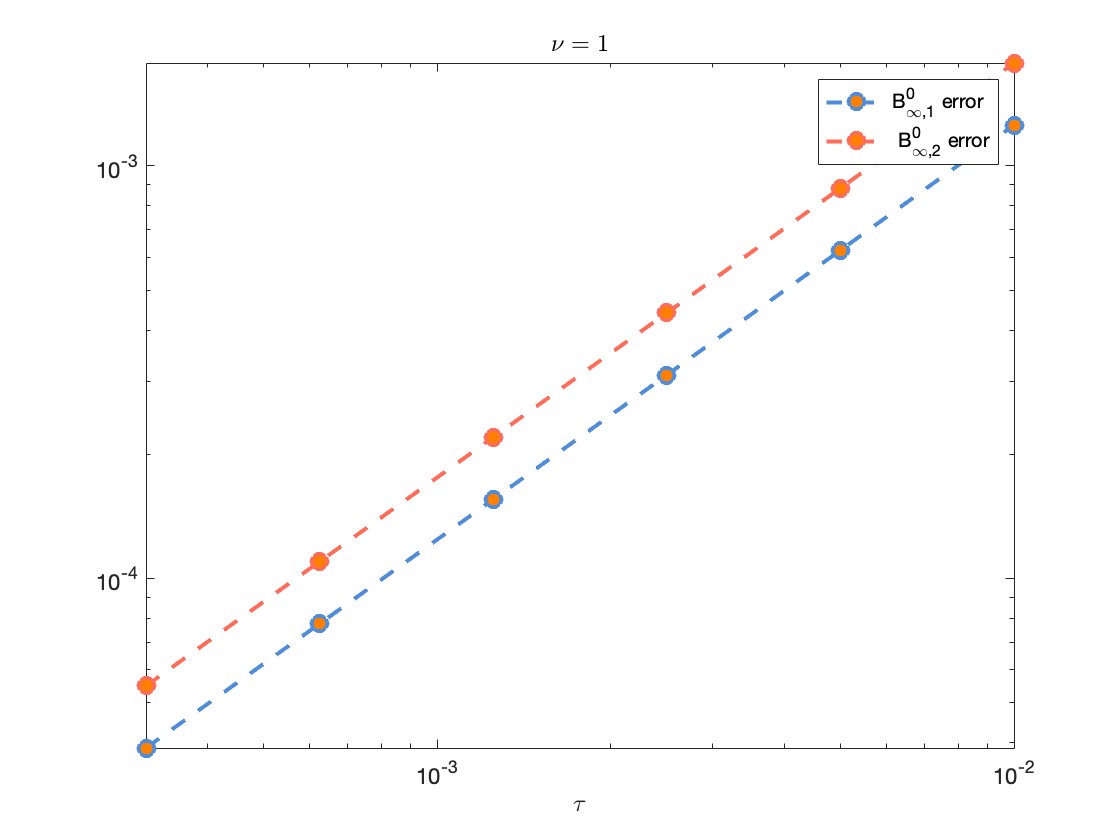}
 \captionof{figure}{Errors with $\nu=1$.}\label{fig1}
\end{minipage}
\hfill
\begin{minipage}[htb]{0.48\linewidth}
    \centering
\includegraphics[width=1\textwidth]{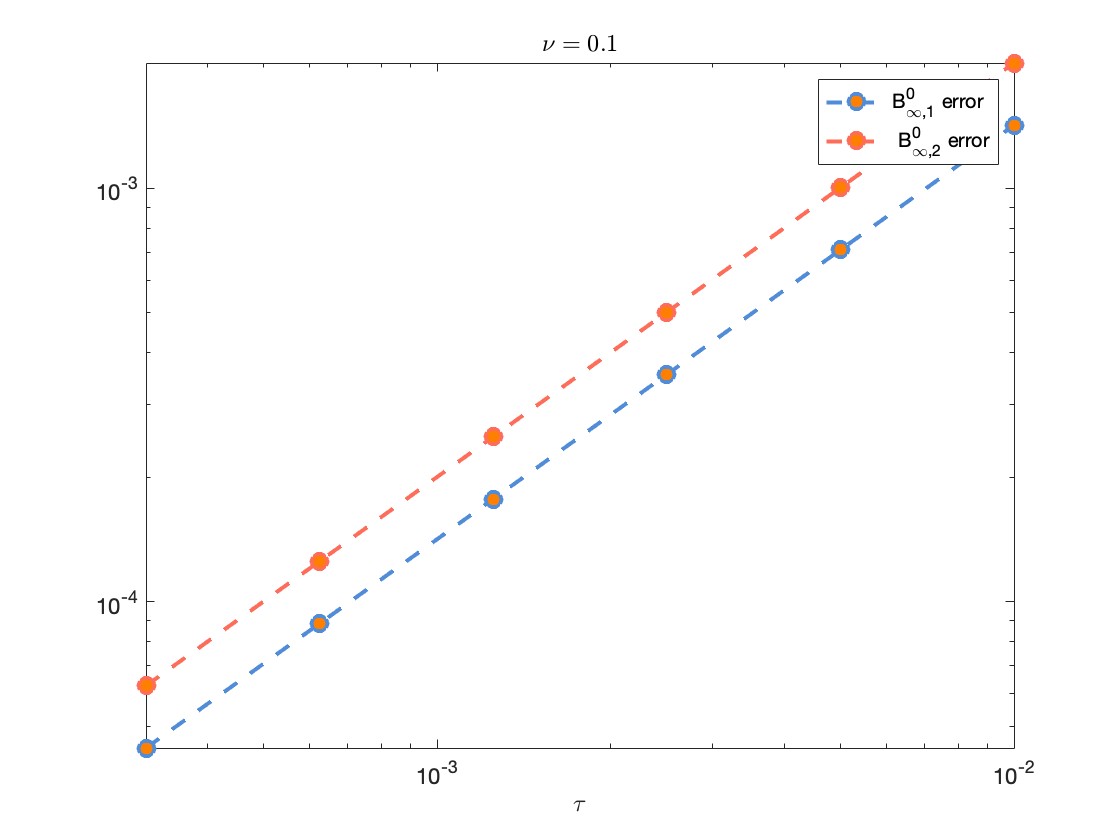}
 \captionof{figure}{Errors with $\nu=0.1$.}\label{fig2}
\end{minipage}
\end{figure}

\begin{figure}[htb]
\begin{minipage}[htb]{0.48\linewidth}
    \centering
\includegraphics[width=1\textwidth]{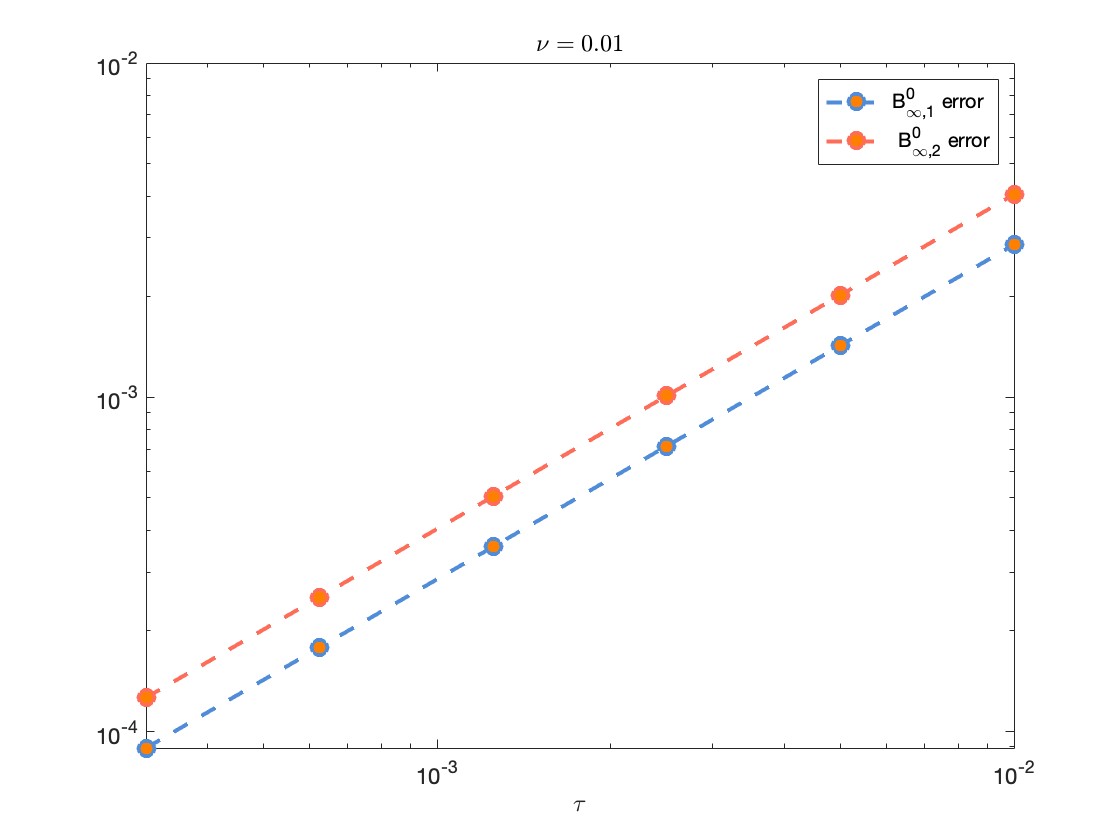}
\captionof{figure}{Errors with $\nu=0.01$.}\label{fig3}
\end{minipage}
\hfill
\begin{minipage}[htb]{0.48\linewidth}
    \centering
\includegraphics[width=1\textwidth]{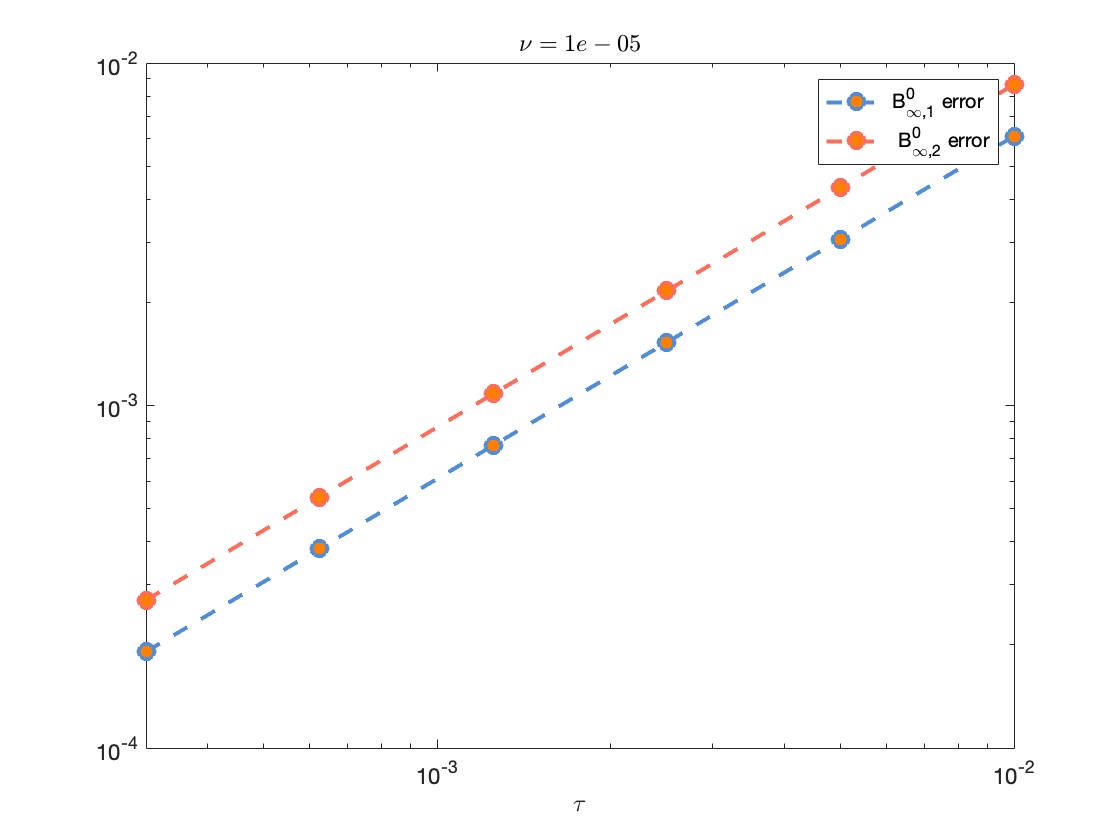}
\captionof{figure}{Errors with $\nu=0.00001$.}\label{fig4}
\end{minipage}

\end{figure}

\subsection{Vanishing viscosity rate}
In this subsection we show the rate of vanishing viscosity using the scheme \eqref{1.3} or \eqref{1.5}. We again adopt the explicit solution $\vec{u_e}$ defined in \eqref{ue} and solve the scheme \eqref{7.2}; we shall compare the numerical solutions of \eqref{7.2} to the 
reference solution under different choices of $\nu$-values at $T=0.1$ with fixed $\tau=0.0001$. The $B^0_{\infty,1}$ and $B^0_{\infty,2}$-differences are given in Table~\ref{table5} and Figure~\ref{fig5}. We can conclude that they are of order $O(\nu)$.

\begin{table}[htb]
\centering
\begin{tabular}{c  c c c c} 
 \toprule 
 $\nu=0.1$ & $L^2$-difference & $B^0_{\infty,1}$-difference & $B^0_{\infty,2}$-difference \\ [0.5ex]
 \hline
 $\nu$ &0.0418 & 0.0188 & 0.0133  \\[0.5ex]
 
 $\nu/2$ & 0.0210 & 0.0095 &0.0067
 \\[0.5ex]
$\nu/4$ &  0.0105& 0.0047 &0.0034
 \\[0.5ex]
 $\nu/8$ & 0.0053 &0.0024 &0.0017
 \\[0.5ex]
 $\nu/16$ & 0.0026 &0.0012 & 8.431e-04
 \\[0.5ex]
 $\nu/32$ & 0.0013 & 5.987e-04 &4.234e-04\\
\bottomrule 
\end{tabular}
\caption{Rate of the vanishing viscosity}
\label{table5}
\end{table}

\begin{figure}[htb]
    \centering
\includegraphics[width=0.5\textwidth]{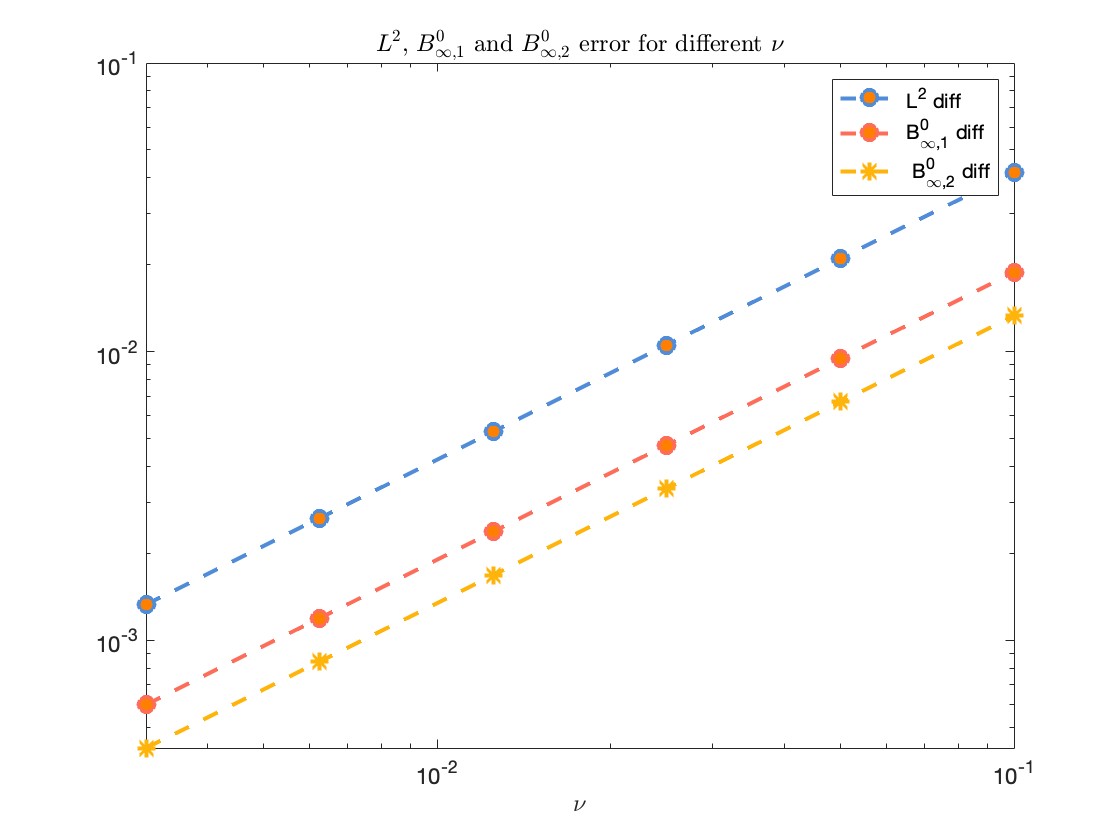}
 \caption{Rate of the vanishing viscosity}
\label{fig5}
\end{figure}

\section{Concluding remark}

In this paper, we have established a new Besov numerical framework for the Navier-Stokes and Euler equations, proving rigorous error estimates for a semi-implicit scheme in the $ B^0_{\infty,1}$ and $ B^0_{\infty,2}$ spaces. Our analysis demonstrates that this approach yields sharper, more localized convergence results than those attainable in classical Sobolev spaces, effectively bridging a gap between the analytical functional setting of the continuous equations and their discrete counterparts. A pivotal element of this work is the refined treatment of the $B^0_{\infty,2}$ case, where a targeted integration-by-parts technique was essential for managing the nonlinear advection term. This technique allowed for a precise transfer of derivatives, enabling a stable error estimate that directly reflects the critical regularity of the true solution. This methodological innovation not only resolves a key technical hurdle in low-regularity spaces but also provides a blueprint for analyzing other nonlinear fluid PDEs where similar terms pose analytical challenges.

This Besov-based framework opens several promising directions for future research. Immediate extensions include adapting the analysis to more complex models, such as the Navier-Stokes-Q-tensor systems, where interfacial dynamics and coupling phenomena could benefit from this finer-grained error control. Ultimately, this work underscores the significant potential of leveraging analytical function spaces in numerical analysis to develop more robust and theoretically sound computational methods.

\section*{Acknowledgement}
X. Cheng is partially supported by NSFC (Grants No. 12401270 and No. 42450192), Natural Science Foundation of Shanghai (Grant No. 24ZR1404200) and Shanghai Magnolia Talent Plan Pujiang Project (Grant No. 24PJA007). S. Wang is partially supported by the China National Postdoctoral Program for Innovative Talents (Grant No. BX20250061) and China Postdoctoral Science Foundation (Grant No. 2024M760058).

\bibliographystyle{abbrv}

\begin{thebibliography}{10}



\bibitem{BG02}
I.~Babuska and B.~Guo. Direct and inverse approximation theorems for the p-version of the finite element method in the framework of weighted Besov spaces. Part I: Approximability of functions in the weighted Besov spaces. {\em SIAM Journal on Numerical Analysis}, 39(5): 1512-1538, 2002.

\bibitem{BCD11}
H.~Bahouri, J.~Y.~Chemin and R.~Danchin. Fourier analysis and nonlinear partial differential
equations, {\em Springer, Heidelberg}, 2011.

\bibitem{BLW22}
G.~Bai, B.~Li and Y.~Wu.
\newblock{A constructive low-regularity integrator for the 1d cubic nonlinear Schr\"odinger equation under the Neumann boundary condition.}
\newblock{\em IMA J. Numer. Anal.}, 43(6): 3243–3281, 2023.

\bibitem{BKM84}
J.~T.~Beale, T.~Kato and A.~Majda. Remarks on the Breakdown of Smooth Solutions for the 3D Euler Equations. {\em Comm. Math. Phys.}, 94(1):61–66 1984.

\bibitem{BDD04}
P.~Binev, W.~Dahmen and R.~DeVore. Adaptive Finite Element Methods with convergence rates. {\em Numer. Math.}, 97: 219–268, 2004. 



\bibitem{BL15}
J.~Bourgain, D.~Li. Strong ill-posedness of the Incompressible Euler Equation in Borderline Sobolev Spaces. {\em Invent. Math.}, 201(1):97–157, 2015.

 \bibitem{BL15b}
 J. Bourgain and D. Li. Strong illposedness of the incompressible Euler equation in integer $C^m$ spaces.{\em Geom. Funct. Anal.}, 25: 1-86, 2015.

\bibitem{BP08}
J.~Bourgain and N.~Pavlovic. Ill-posedness of the Navier-Stokes Equations in a Critical
Space in 3D. {\em J. Funct. Anal.}, 255(9):2233–2247 2008.

\bibitem{BSV19}
T.~Buckmaster, S.~Shkoller and V.~Vicol. Nonuniqueness of weak solutions to the SQG equation. {\em Commun.
Pure Appl. Math.}, 72(9): 1809–1874, 2019.






\bibitem{CKL21}
X.~Cheng, H.~Kwon and D.~Li. Non-uniqueness of stationary weak solutions to the surface quasi-geostrophic equations. {\em Comm. Math. Phys.}, 388 (3): 1281-1295, 2021.

\bibitem{CLW24}
X.~Cheng, Z.~Luo and S.~Wang.
On semi-implicit schemes for the incompressible Euler equations via the vanishing viscosity limit. Preprint, arXiv:2406.12320.


\bibitem{CLW25}
X.~Cheng, Z.~Luo and S.~Wang. A sharp low regularity method for the surface quasi-geostrophic equations: a Besov framework. Preprint, submitted.


\bibitem{CLYY24}
X.~Cheng, Z.~Luo, Z.~Yang and C.~Yuan. Global well-posedness and uniform-in-time vanishing damping limit for the inviscid Oldroyd-B model, preprint, arXiv:2410.09340.




\bibitem{Chorin68}
A.J.~Chorin. Numerical solution of the Navier–Stokes equations. {\em Math. Comput.}, 22, 745–762, 1968.

\bibitem{Chorin69}
A.J.~Chorin. On the convergence of discrete approximations to the Navier–Stokes equations. {\em Math. Comput.}, 23, 341–353, 1969.


\bibitem{DLS13}
C.~De Lellis and L.~Sz\'ekelyhidi, Jr. Dissipative continuous Euler flows. {\em Invent. Math.}, 193(2): 377–407, 2013.



\bibitem{EL95}
W.~E and J.~Liu. Projection method I: convergence and numerical boundary layers. {\em SIAM journal on numerical analysis}, 1017-1057, 1995.
\bibitem{EL02}
W.~E and J.~Liu. Projection method III: spatial discretization on the staggered grid. {\em Mathematics of computation}, 71(237): 27-47, 2002.

\bibitem{Gan17}
T.~Gantumur. Convergence rates of adaptive methods, Besov spaces, and multilevel approximation. {\em Foundations of Computational Mathematics}, 17(4): 917-956, 2017.

\bibitem{Guer04}
J. L.~Guermond. A finite element technique for solving first-order PDEs in $L^p$. {\em SIAM Journal on Numerical Analysis}, 42(2): 714-737, 2004.

\bibitem{GZ03}
B.~Guo and J.~Zou. Fourier spectral projection method and nonlinear convergence analysis for Navier–Stokes equations. {\em Journal of mathematical analysis and applications}, 282(2): 766-791, 2003.

\bibitem{GLY19}
Z.~Guo, J.~Li and Z.~Yin. Local Well-posedness of the Incompressible Euler Equations in $B^1_{\infty,1}$
and the Inviscid Limit of the Navier-Stokes Equations. {\em J. Funct. Anal.}, 276(9): 2821–2830, 2019.


\bibitem{HJE18}
J.~Han, A.~Jentzen and W.~E. Solving high-dimensional partial differential equations using deep learning, {\em Proc. Natl. Acad. Sci. U.S.A.} 115(34): 8505-8510, 2018.


\bibitem{He08}
Y.~He. The Euler implicit/explicit scheme for the 2D time-dependent Navier-Stokes equations with
smooth or non-smooth initial data. {\em Math. Comp.}, 77(264): 2097–2124, 2008.

\bibitem{He13}
Y.~He. Euler implicit/explicit iterative scheme for the stationary Navier–Stokes equations. {\em Numer. Math.}, 123: 67–96, 2013.

\bibitem{HZ18}
Y.~He and J.~Zou. A priori estimates and optimal finite element approximation of the MHD flow in smooth domains. {\em ESAIM: Mathematical Modelling and Numerical Analysis}, 52(1): 181-206, 2018.

\bibitem{HR82}
J.G.~Heywood and R.~Rannacher. Finite element approximation of the nonstationary Navier–Stokes problem. I. Regularity of solutions and second-order spatial discretization. {\em SIAM J. Numer. Anal}., 19,
275–311, 1982.

\bibitem{HR90}
J.G.~Heywood and R.~Rannacher. Finite-element approximation of the nonstationary Navier–Stokes problem part IV: error analysis for second-order time discretization. {\em SIAM J. Numer. Anal.}, 27, 353–384, 1990.

\bibitem{HS00}
A. T.~Hill and E.~S\"uli. Approximation of the global attractor for the incompressible Navier-
Stokes equations, {\em IMA J. Numer. Anal.}, 20: 633--667, 2000.


\bibitem{HK08}
T.~Hmidi and S.~Keraani. Incompressible viscous flows in borderline Besov spaces. {\em Arch. for
Rational Mech. and Analysis}, 189(2):283–300, 2008.


\bibitem{HW93}
T. Y.~Hou and B.~Wetton.
{Second-Order Convergence of a Projection Scheme for the Incompressible Navier–Stokes Equations with Boundaries},
{\em SIAM Journal on Numerical Analysis},
30(3): 609--629, 1993.



\bibitem{JCLK21}
X. Jin, S. Cai, H. Li and G.E. Karniadakis. NSFnets (Navier-Stokes flow nets): Physics-informed neural networks for the incompressible Navier-Stokes equations. {\em J. Comp. Phys.}, 426: 109951, 2021.

\bibitem{KKLPWY21}
G.E. Karniadakis, I.G. Kevrekidis, L. Lu, P. Perdikaris, S. Wang and L. Yang. Physics-informed machine learning.{\em Nat. Rev. Phys.,} 3: 422–440, 2021. 


\bibitem{LL11}
Z. Lei and F. Lin. 
Global mild solutions of Navier–Stokes equations.
{\em Communications on Pure and Applied Mathematics}, 64(9): 1297-1304, 2011.


\bibitem{L34}
J.~Leray. Sur Le Mouvement Dun Liquide Visqueux Emplissant Lespace. {\em Acta Math.}, 63(1):193–248, 1934.





\bibitem{LQY22}
B.~Li, W.~Qiu, and Z.~Yang. A Convergent Post-processed Discontinuous Galerkin Method for Incompressible Flow with Variable Density. {\em J. Sci. Comput.}, 91(2), 2022.




\bibitem{LMS22}
B.~Li, S.~Ma and K.~Schratz. A Semi-implicit Exponential Low-Regularity Integrator for the Navier--Stokes Equations. {\em SIAM J. Numer. Anal.}, 60(4): 2273-2292, 2022.


\bibitem{LMU22}
B.~Li, S.~Ma and Y.~Ueda. Analysis of fully discrete finite element methods for 2D Navier–Stokes equations with critical initial data. {\em ESAIM: Mathematical Modelling and Numerical Analysis}, 56(6): 2105-2139, 2022.

\bibitem{Li19}
D.~Li. On Kato–Ponce and fractional Leibniz. {\em Rev. Mat. Iberoam.}, 35(1): 23–100, 2019.




\bibitem{MT98}
 M.~Marion and R.~Temam. Navier-Stokes equations: Theory and approximation, in Numerical
Methods for Solids (Part 3), Numerical Methods for Fluids (Part 1), {\em Handb. Numer. Anal.
6, Elsevier--North Holland}, Amsterdam, 503--689, 1998.

\bibitem{M07}
N.~Masmoudi. Remarks about the inviscid limit of the Navier-Stokes system, {\em Commun. Math. Phys.}, 270(3): 777–788, 2007.




\bibitem{RS21}
F.~Rousset and K.~Schratz. A general framework of low regularity integrators, {\em SIAM J.
Numer. Anal.}, 59(3): 1735--1768, 2021.




\bibitem{S88}
 E.~S\"uli. Convergence and non-linear stability of the Lagrange–Galerkin method for the Navier–Stokes equations. {\em Numer. Math.}, 53, 459–483, 1988.


\bibitem{Tem66}
R.~Temam. Sur l'approximation des solutions des\'equations de Navier-Stokes, {\em C.R. Acad. Sci.
Paris, Serie A}, 262, 219-221, 1966.


\bibitem{Tem24}
R.~Temam. Navier–Stokes equations: theory and numerical analysis. {\em American Mathematical Society}, 2024.


\bibitem{V84}
R.~Verf\"urth. Error estimates for a mixed finite element approximation of the Stokes equations. {\em RAIRO
Anal. Numer.} 18(2), 175–182, 1984.

\bibitem{Wang12}
 X.~Wang. An efficient second order in time scheme for approximating long time statistical properties of the two dimensional Navier–Stokes equations. {\em Numer. Math.}, 121: 753–779, 2012.

\bibitem{Wetton06}
B.~Wetton. Finite difference vorticity methods. {\em The Navier-Stokes Equations II—Theory and Numerical Methods: Proceedings of a Conference held in Oberwolfach, Germany, August 18–24, 1991.} Berlin, Heidelberg: Springer Berlin Heidelberg, 210-225, 2006.


\bibitem{WZ22}
Y. Wu and X. Zhao. Optimal convergence of a second order low-regularity integrator for the KdV equation.
{\em IMA J. Numer. Anal.}, 42(4): 3499–3528, 2022.

\bibitem{Y63}
V.~I.~Yudovich. Non-stationary Flows of an Ideal Incompressible Fluid. {\em Z. Vycisl.
Mat i Mat. Fiz.}, 3:1032–1066, 1963.

\end{thebibliography}

\end{document}